\documentclass[reqno]{amsart}

\newtheorem{theorem}{Theorem}[section]
\newtheorem{lemma}[theorem]{Lemma}
\newtheorem{proposition}[theorem]{Proposition}
\newtheorem{corollary}[theorem]{Corollary}

\theoremstyle{definition}
\newtheorem{definition}[theorem]{Definition}

\theoremstyle{remark}
\newtheorem{remark}[theorem]{Remark}

\numberwithin{equation}{section}

\usepackage{graphics}
\usepackage{mathtools}
\usepackage{bbm}
\usepackage[utf8]{inputenc} 
\usepackage[T1]{fontenc}
\usepackage{float}
\usepackage{amssymb}
\usepackage{srcltx}
\usepackage{enumerate}
\usepackage{esint}
\usepackage[shortlabels]{enumitem}
\usepackage{gensymb}

\usepackage[colorlinks,plainpages=true,pdfpagelabels,hypertexnames=true,colorlinks=true,pdfstartview=FitV,linkcolor=blue,citecolor=red,urlcolor=black]{hyperref}

\everymath{\displaystyle}

\newcommand{\abs}[1]{\left\lvert#1\right\rvert}
\newcommand{\norm}[1]{\abs{\abs{#1}}}

\newcommand{\co}{\colon}

\newcommand{\defeq}{\vcentcolon=}
\DeclarePairedDelimiterX{\inp}[2]{\langle}{\rangle}{#1, #2}

\newcommand\boldbox[1]{\mbox{\boldmath{\boxed{\ensuremath #1}} \unboldmath}}

\DeclareMathOperator{\st}{{\bf st}}
\DeclareMathOperator{\Prob}{{\bf Prob}}

\newcommand{\starint}{{\prescript{\ast}{}\int}}

\usepackage[dvipsnames]{xcolor}

\newcommand{\mbp}{{\mathbb P}}

\newcommand{\mbx}{{\mathbf X}}

\makeatletter
\newtheorem*{rep@theorem}{\rep@title}
\newcommand{\newreptheorem}[2]{%
\newenvironment{rep#1}[1]{%
 \def\rep@title{#2 \ref{##1}}%
 \begin{rep@theorem}}%
 {\end{rep@theorem}}}
\makeatother

\newreptheorem{theorem}{Theorem}

\newcommand\mystrut{\rule{0pt}{15pt}}
\newcommand\mystrutt{\rule{0pt}{20pt}}

\begin{document}

\title{Limiting Probability Measures}


\author{Irfan Alam}
\address{Department of Mathematics, Louisiana State University, Baton Rouge, LA, USA.}
\curraddr{}
\email{irfanalamisi@gmail.com}
\thanks{}



\date{}

\dedicatory{}

\commby{}

\begin{abstract}
The coordinates along any fixed direction(s), of points on the sphere $S^{n-1}(\sqrt{n})$, roughly follow a standard Gaussian distribution as $n$ approaches infinity. We revisit this classical result from a nonstandard analysis perspective. With that goal, we first set up a nonstandard theory for the asymptotic behavior of integrals over varying domains in general. We obtain a new proof of the Riemann-Lebesgue lemma as a by-product of this theory. We then define an appropriate surface area measure on a hyperfinite-dimensional sphere and show that for any function $f: \mathbb{R}^k \rightarrow \mathbb{R}$ with finite Gaussian moment of an order larger than one, its expectation is given by an integral over this sphere. Some useful inequalities between high-dimensional spherical means of $f$ and its Gaussian mean are obtained in order to complete the above proof. A review of the requisite nonstandard analysis is provided. 
\end{abstract}

\maketitle



\begin{section}{Introduction}
Gaussian measures have been mathematically connected with the uniform surface area measures on high-dimensional spheres since at least the time of Poincar\'e, who observed in \cite{Poincare} that if $n$ real numbers are randomly chosen under the constraint that their sum of squares equals $n$ (this is equivalent to choosing a random vector on $S^{n-1}(\sqrt{n})$, the sphere in $\mathbb{R}^n$ centered at the origin, of radius $\sqrt{n}$), then as $n \rightarrow \infty$, the probability distribution of the first number converges to that of a standard Gaussian random variable (that is, with zero mean and covariance equaling one). Considering works on the kinetic theory of gases in physics, this connection goes back another century (we briefly outline this connection with Physics in Appendix \ref{appendix}).

For any sphere $S$ centered at the origin in a Euclidean space, there is a unique orthogonal transformation invariant probability measure $\bar{\sigma}_S$ (we will omit the subscript when the sphere under consideration is clear from context). We refer to Matilla \cite[Chapter 3]{spherical_measures} for basic properties of spherical surface area measures. For each $k \in \mathbb{N}$ and $n \in \mathbb{N}_{\geq k}$, let $\pi^{(n)}_k : \mathbb{R}^n \rightarrow \mathbb{R}^k$ denote the projection on to the first $k$ coordinates under the standard basis (we will omit the superscript when the dimension is clear from context). For a Borel set $B \subseteq \mathbb{R}^k$, we write $$\bar{\sigma}_{S^{n-1}(\sqrt{n})}(B) \defeq \bar{\sigma}_{S^{n-1}(\sqrt{n})}[S^{n-1}(\sqrt{n}) \cap (\pi^{(n)}_k)^{-1}(B)].$$ 

In the same spirit, we identify each measurable function $f: \mathbb{R}^k \rightarrow \mathbb{R}$ with a function on $\mathbb{R}^n$ by composing it with the projection $\pi^{(n)}_k$. 

We let $\mu_{(k)}$ denote the standard Gaussian measure on $\mathbb{R}^k$ (again, omitting the subscript when the dimension is clear). With these conventions, we may write Poincar\'e's observation succinctly as
\begin{align}\label{Poincare's first limit}
    \lim_{n \rightarrow \infty} \bar{\sigma}_{S^{n-1}(\sqrt{n})}(B) = \mu(B) \text{ for all Borel sets } B \subseteq \mathbb{R}.
\end{align}

By standard measure theory, it is not difficult to see that the above can be rephrased in a more general form as follows (as discussed above, the integral on the left side will be understood as that of the function $f \circ \pi^{(n)}_k$ for all $n \in \mathbb{N}_{>k}$).

\begin{theorem}[Poincar\'e, \cite{Poincare}]\label{Poincare's theorem}
For all bounded measurable functions $f: \mathbb{R}^k \rightarrow \mathbb{R}$, we have 
\begin{align}\label{Poincare limit}
\lim_{n \rightarrow \infty} \int_{S^{n-1}(\sqrt{n})} f d\bar{\sigma} = \int_{\mathbb{R}^k} f d\mu.    
\end{align}
\end{theorem}

Similar ideas were later used by L\'evy \cite{Levy} to do infinite dimensional analysis, and then by Wiener \cite{Wiener} to construct Brownian motion. McKean \cite{McKean} surveyed most of the relevant work from that period. Cutland and Ng explored these themes using nonstandard analysis (which provides the language of hyperfinite dimensional spheres) in \cite{Cutland-Ng}. They gave a new construction of the Wiener measure using the nonstandard machinery. 

The current paper may be considered a sequel to \cite{Cutland-Ng} in some sense. Indeed one of our aims is to view the above classical result (Theorem \ref{Poincare's theorem}) as a statement about Loeb integrals on hyperfinite dimensional spheres, and obtain the same result for a larger class of functions. Toward that end, we give a new nonstandard proof of Poincar\'e's theorem in Section \ref{nonstandard Poincare}. A novel feature of this proof is that it does not require any explicit integral calculations -- it follows from straightforward applications of the weak law of large numbers and the definition of the uniform surface area measure on a sphere as a pushforward of a Gaussian measure. In Section \ref{section 3}, we also establish a nonstandard approach of extending such results from bounded measurable functions to other classes of functions. The general framework described in Sections \ref{section 2} and \ref{section 3} may be thought of as an invitation to apply nonstandard analysis to other asymptotic problems in probability and measure theory. One such application is carried out in Alam \cite{Alam2} to generalize recent works of Sengupta \cite{Sengupta} and Peterson--Sengupta \cite{Sengupta-Peterson} that connect Gaussian Radon transforms with limiting spherical integrals. This generalization is the topic of one of the author's ongoing projects. 

We also give a classical standard proof of Theorem \ref{Poincare's theorem} in Section \ref{computational proof} -- it follows by dominated convergence theorem once the integral over the sphere is ``disintegrated'' properly (for example, using Sengupta \cite[Proposition 4.1]{Sengupta}). As pointed out in Remark \ref{remark on dct not working}, this proof of Theorem \ref{Poincare's theorem} does not immediately generalize to work for an arbitrary $\mu$--integrable function. The nonstandard framework of Section 3 allows one to get conditions (see Theorems \ref{sufficient condition for integrability} and \ref{main equivalence}) under which a result of the type of Theorem \ref{Poincare's theorem} for bounded measurable functions (over general domains) can be extended to unbounded functions. Though we do not use this terminology, the framework in Section \ref{section 3} is similar to the framework of \textit{graded probability spaces}, as in Hoover \cite{Hoover-IAS} and Keisler \cite{Keisler-hyperfinite}.

Aside from its application to spherical integrals, the approach of Section \ref{section 4} is potentially useful in many other situations in which limits of integrals may be studied. A new proof of the Riemann-Lebesgue Lemma is provided (see Theorem \ref{Riemann-Lebesgue}) as an example of its use. Finally, in order to verify the sufficient conditions from Section 4 in the case of spherical integrals, we also prove some inequalities between spherical means and $L^p(\mathbb{R}^k, \mu)$ norms of functions on $\mathbb{R}^k$ (see Theorem \ref{p>1} and Corollary \ref{p>1 part 2}). Thus, the main results of this paper can be divided into three types:
\begin{itemize}
    \item Results viewing the limiting behavior of integrals over varying abstract domains as a single integral over a nonstandard domain. 
    
    \item Inequalities between spherical integrals and Gaussian integrals.
    \item Applications of the results of the above types to systematically generalize Theorem \ref{Poincare's theorem} on limiting spherical integrals to a bigger class of functions.
\end{itemize}

\subsection{Summary and motivation of our key results} Recall that for a Borel measurable function $f: \mathbb{R}^k \to \mathbb{R}$, we are interested in $$\lim_{n \rightarrow \infty} \int_{S^{n-1}(\sqrt{n})} f(x_1, \ldots, x_k) d\bar{\sigma}(x_1, \ldots, x_n),$$ where we view $f$ as a function on $\mathbb{R}^n$ by first projecting the input into the first $k$ coordinates. Assuming Theorem \ref{Poincare's theorem}, if $f$ is bounded, then we know from \eqref{Poincare limit} that this limit is equal to the expected value of $f$ with respect to the standard Gaussian measure $\mu$ on $\mathbb{R}^k$. Since we are assuming the limiting result $\eqref{Poincare limit}$ for bounded functions, we have (using $\mathbbm{1}_B$ to denote the indicator function of a set $B$) the following for a possibly unbounded Borel measurable function $f: \mathbb{R}^k \to \mathbb{R}$.
\begin{align}\label{double limits}
     \lim_{m \rightarrow \infty} \lim_{n \rightarrow \infty} \int_{S^{n-1}(\sqrt{n})} f\mathbbm{1}_{\abs{f} \leq m} d\bar{\sigma} &= \lim_{m \rightarrow \infty} \int_{\mathbb{R}^k} f\mathbbm{1}_{\abs{f} \leq m} d\mu =  \int_{\mathbb{R}^k} f d\mu.
\end{align}

However, we wanted to find $\lim_{n \rightarrow \infty} \int_{S^{n-1}(\sqrt{n})} f d\bar{\sigma}$, which (assuming that $f$ is integrable over $S^{n-1}(\sqrt{n})$ for large $n \in \mathbb{N}$) is the same as the following: $$\lim_{n \rightarrow \infty} \lim_{m \rightarrow \infty} \int_{S^{n-1}(\sqrt{n})} f\mathbbm{1}_{\abs{f} \leq m} d\bar{\sigma}.$$

Thus, in order to go from a result on bounded functions to a result on more general functions, we want to be able to switch the order of limits in \eqref{double limits}. However, there is no general theory of switching double limits.

From the point of view of nonstandard analysis (which we shall review in the next section), the situation is simpler since the large-$n$ behavior of any sequence is captured in the values attained by the nonstandard extension of that sequence at hyperfinite indices. For a hyperfinite $N > \mathbb{N}$, the sphere $S^{N-1}(\sqrt{N})$ inherits a finitely additive internal probability measure from the sequence $(S^{n-1}(\sqrt{n}), \bar{\sigma}_{S^{n-1}(\sqrt{n})})_{n \in \mathbb{N}}$. The $N^{\text{th}}$ term in the nonstandard extension of the sequence $\left( \int_{S^{n-1}(\sqrt{n})} f d\bar{\sigma}\right)_{n \in \mathbb{N}}$ is then the ${^*}$--integral of ${^*}f$ with respect to this internal measure. It turns out that the limiting integral for a general measurable function $f\co \mathbb{R}^k \to \mathbb{R}$ exists (knowing that it exists and is equal to the Gaussian mean for bounded measurable functions) if ${^*}f$ is $S$--integrable over $S^{N-1}(\sqrt{N})$. In a more abstract setting, Theorem \ref{sufficient condition for integrability} essentially tells us that we can switch these limits if the tail double-limit $\lim_{m \rightarrow \infty} \lim_{n \rightarrow \infty} \int_{S^{n-1}(\sqrt{n})} \abs{f} \mathbbm{1}_{\abs{f} > m} d\bar{\sigma}$ is zero. This condition of the tail double-limit being zero is just a standard reformulation of one of the equivalent conditions that ensure the $S$-integrability of ${^*}f$ over $S^{N-1}(\sqrt{N})$ (see \ref{S2} of Theorem \ref{S-integrable TFAE}). 

A partial converse of the above result holds for nonnegative functions, which is covered in Theorem \ref{main equivalence}. Thus the set of all nonnegative functions for which the limit of spherical integrals is equal to the Gaussian integral is precisely the set of nonnegative functions for which the above tail double-limit is zero. While Theorems \ref{sufficient condition for integrability} and \ref{main equivalence} come out of nonstandard measure theoretic considerations, we paraphrase a standard version for convenience as follows:

\begin{theorem}\label{Summary of main equivalence theorem}
Let $(E, \mathcal{E})$ be a measure space. Let $k \in \mathbb{N}$ and for each $n \in \mathbb{N}_{> k}$, suppose $\Omega_n \subseteq E^{n'}$ for some $n' \in \mathbb{N}_{> k}$. Suppose that $\mathcal{F}_n$, the given sigma-algebra on $\Omega_n$, is induced by the product sigma-algebra $\mathcal{E}_{n'}$ on $E^{n'}$. Let $(\Omega_n, \mathcal{F}_n, \nu_n)$ be a sequence of Borel probability spaces. Let $\mathbb{P}$ be a probability measure on $(E^k, \mathcal{E}_k)$ such that
$\lim_{n \rightarrow \infty} \nu_n (B) = \mathbb{P}(B) \text{ for any } B \in \mathcal{E}_k$. Then \eqref{good double limit intro} implies \eqref{good limit inro} below:
\begin{enumerate}
      \item\label{good double limit intro} A function $f$ is integrable on $(\Omega_n, \nu_n)$ for all large $n \in \mathbb{N}$, and furthermore:
    $$ \lim_{m \rightarrow \infty} \lim_{n \rightarrow \infty} \int_{\Omega_n \cap \{\abs{f} \geq m\}} \abs{f} d{\nu}_n = 0.$$
    
      \item\label{good limit inro} The function $f\co E^k \rightarrow \mathbb{R}$ is $\mathbb{P}$-integrable and $\lim_{n \rightarrow \infty} \int_{\Omega_n} f d\nu_n = \int_{E^k} f d\mathbb{P}$.
\end{enumerate}

Furthermore, if $f$ is assumed to be nonnegative, then the above conditions \eqref{good double limit intro} and \eqref{good limit inro} are equivalent. 
\end{theorem}

The above theorem can also be interpreted more classically as a statement involving uniform integrability. While we do not focus on this aspect, it is interesting to emphasize  that the nonstandard arguments using $S$--integrability thus encompass standard uniform integrability techniques.   

In the case when $\Omega_n$ are the spheres $S^{n-1}(\sqrt{n})$, we verify the above double limit condition for all functions on $\mathbb{R}^k$ with a finite $(1 + \epsilon)$--Gaussian moment, where $\epsilon$ is any positive real number. This allows us to extend the result in Theorem \ref{Poincare's theorem} to all such functions (see Theorem \ref{most general}). The main step in this verification is an inequality (see Theorem \ref{p>1} and Corollary \ref{p>1 part 2}) between sufficiently high-dimensional spherical means and $L^p(\mathbb{R}^k, \mu)$ norms of functions on $\mathbb{R}^k$, which we summarize as follows:

\begin{theorem}\label{p>1 introduction}
For each $p \in \mathbb{R}_{>1}$, there is a constant $C_p \in \mathbb{R}_{>0}$ such that the following holds.
\begin{align}
    \int_{S^{n-1}(\sqrt{n})} \abs{g} d\bar{\sigma}_n \leq C_p[\mathbb{E}_\mu(\abs{g}^p)]^{\frac{1}{p}} \text{ for all } g \in L^p(\mathbb{R}^k, \mu) \text{ and } n \in \mathbb{N}_{> 4(k+1)q},
\end{align}
where $q \in \mathbb{R}_{>0}$ is such that $\frac{1}{p} + \frac{1}{q} = 1$.

Furthermore, we may replace the constant $C_p$ in the above inequality by a real number as close to $1$ as desired if $n$ is taken large enough (this large $n$ depends only on $p \in \mathbb{R}_{>1}$ and the desired distance of the constant from $1$).
\end{theorem}

\subsection{Structure of the paper}
The next section presents the background from nonstandard analysis that we will use. We attempt to present the material in a way that requires very little formal training in Logic. The interested reader is directed toward the numerous good books on nonstandard analysis for more background (see \cite{Albeverio}, \cite{Cutland}, and \cite{Robinson}). 

Section \ref{section 2} contains a nonstandard proof of Theorem \ref{Poincare's theorem} (carried out in Section \ref{nonstandard Poincare}), which is prefaced by some basic nonstandard measure theory that we will use and a discussion on spherical measures (alongwith their nonstandard counterparts). 

Section \ref{section 3} continues the theme by studying sequences of abstract measure spaces for which a result of the type of Poincar\'e is known. It gives conditions under which such results hold for more general functions, allowing us to express the limiting behavior of certain integrals by a Loeb integral on a single limiting measure space. An application that yields a new proof of the Riemann-Lebesgue lemma is carried out in Theorem \ref{Riemann-Lebesgue}.

In Section \ref{section 4}, we apply the results of Section 3 to the case of high-dimensional spheres, and obtain a generalization of the classical result on limits of spherical integrals to a large class of Gaussian integrable functions (see Theorem \ref{most general}). Toward that end, we also obtain some useful inequalities between high-dimensional spherical means and Gaussian means (see Theorem \ref{p>1} and Corollary \ref{p>1 part 2}). 
\end{section}

\begin{section}{A non-logical crash course in nonstandard analysis}
\subsection{Basic nonstandard analysis}
There are many approaches to nonstandard analysis, eight of which were described in Benci--Forti--di Nasso \cite{eightfold}. We follow the superstructure approach, as done in Albeverio et al. \cite{Albeverio}. Roughly, a nonstandard extension of a set $\mathfrak{X}$ (consisting of atoms or urelements; that is, we view each element of $X$ as an ``individual'' without any structure, set-theoretic or otherwise) is a superset ${^*}\mathfrak{X}$ that preserves the ``first-order'' properties of $\mathfrak{X}$. That is, a property which is expressible using finitely many symbols without quantifying over any collections of subsets of $\mathfrak{X}$ is true if and only if the same property is true of ${^*}\mathfrak{X}$. This is called the \textbf{transfer principle} (or just \textit{transfer} for brevity). The set ${^*}\mathfrak{X}$ should contain, as a subset, ${^*}\mathfrak{Y}$ for each $\mathfrak{Y} \subseteq X$. 

Like subsets, other mathematical objects defined on $\mathfrak{X}$ also have extensions. So, a function $f: \mathfrak{X} \rightarrow \mathfrak{Y}$ extends to a map ${^*}f: {^*}\mathfrak{X} \rightarrow {^*}\mathfrak{Y}$, and relations on $\mathfrak{X}$ extend to relations on ${^*}\mathfrak{X}$. Hence there is a binary relation ${^*}<$ on ${^*}\mathbb{R}$, which we still denote by $<$ (an abuse of notation that we will frequently make), and which is the same as the usual order when restricted to $\mathbb{R}$. Thus, ${^*}\mathbb{R}$ is an ordered field of characteristic $0$. Indeed all (the infinitely many) axioms for ordered fields of characteristic $0$ hold for it by transfer. The symbols in a sentence such as $``\forall x > 0 ~\exists y (x = y^2)"$ (which is expressing the proposition that each positive number has a square root) have new meanings in the nonstandard universe: by $``<"$, we are now interpreting the extension of the order on $\mathbb{R}$. Yet the sentence is true in ${^*}\mathbb{R}$ by transfer! 

A non-first-order property of $\mathfrak{X}$ may not transfer to ${^*}\mathfrak{X}$. For instance, we shall see (cf. Proposition \ref{existence of infinites}) that any ``non-trivial'' extension of $\mathbb{R}$ contains \textit{infinite} elements (that is, those that are larger than all real numbers in absolute value), as well as \textit{infinitesimal} elements (that is, those that are smaller than all positive real numbers in absolute value). Thus, the Archimedean property of $\mathbb{R}$ does not transfer. The set of finite nonstandard real numbers, denoted by ${^*}\mathbb{R}_{\text{fin}}$, is a subring of the non-Archimedean field ${^*}\mathbb{R}$. To see what went wrong, note that the following sentences formally express the Archimedean property for $\mathbb{R}$ and its transfer, respectively:
\begin{align}
\label{Archimedean} \forall x \in \mathbb{R} \quad \quad ~&\exists \ n \in \mathbb{N} ~(n > x). \\
\label{*-Archimedean}\forall x \in {^*}\mathbb{R} \quad \quad ~&\exists\  n \in {^*}\mathbb{N} ~(n > x).
\end{align}

The transferred sentence (\ref{*-Archimedean}) no longer expresses the Archimedean property (though it still expresses an interesting fact about ${^*}\mathbb{R}$). The issue is that we are only able to quantify over ${^*}\mathbb{N}$ (and not on $\mathbb{N}$) after transfer. To keep quantifying over $\mathbb{N}$, we would have to transfer an ``infinite statement'' (saying that for every $x$, either $1>x$, or $2>x$, or $3>x$, or $\ldots$), which is not a valid first-order sentence. 

Another non-example is the least upper bound principle: the set $\mathbb{N}$, viewed as a subset of ${^*}\mathbb{R}$, is bounded (by any positive infinite element), yet has no least upper bound (as any upper bound minus one is also an upper bound). The issue here is that the least upper bound property for $\mathbb{R}$ is expressed via the second-order statement:
\begin{align*}
    &\forall A \subseteq \mathbb{R} \\
    &\langle ~{\color{violet}[\exists x \in \mathbb{R}} {\color{OliveGreen}(\forall y \in \mathbb{R}} \{{\color{red}(y \in A) \rightarrow (y \leq x)} \} {\color{OliveGreen})} {\color{violet}]} \color{black}\rightarrow \\
&{\exists z \in \mathbb{R}} \\
&\{{\color{OliveGreen}(\forall y \in \mathbb{R}} ~ [ {\color{red}(y \in A) \rightarrow (y \leq z)} ]{\color{OliveGreen})} \\
&\hspace{3cm} \land {{\color{blue}[\forall w \in \mathbb{R} {\color{OliveGreen} ( \forall y \in \mathbb{R}} {\color{black} \{ [} {\color{Magenta}(y \in A) \rightarrow (y \leq w)} {\color{black}]}} \rightarrow ({\color{brown}z \leq w})\}} {\color{OliveGreen} )} {\color{blue}]}\} ~\rangle. 
\end{align*}

One way to express this as a first-order statement is to quantify over the powerset, $\mathcal{P}(\mathbb{R})$, of $\mathbb{R}$. If our nonstandard map ${^*}$ was able to extend sets of subsets of $X$ as well, then the above would transfer to the following \textit{*-least upper bound property}:

\begin{align*}
    &\forall A \in {^*}\mathcal{P}(\mathbb{R}) \\
    &\langle ~{\color{violet}[\exists x \in {^*}\mathbb{R}} {\color{OliveGreen}(\forall y \in {^*}\mathbb{R}} \{{\color{red}(y \in A) \rightarrow (y \leq x)} \} {\color{OliveGreen})} {\color{violet}]} \color{black}\rightarrow \\
&{\exists z \in {^*}\mathbb{R}} \\
&\{{\color{OliveGreen}(\forall y \in {^*}\mathbb{R}} ~ [ {\color{red}(y \in A) \rightarrow (y \leq z)} ]{\color{OliveGreen})} \\
&\hspace{2.7cm} \land {{\color{blue}[\forall w \in {^*}\mathbb{R} {\color{OliveGreen}( \forall y \in {^*}\mathbb{R}} {\color{black} \{ [} {\color{Magenta}(y \in A) \rightarrow (y \leq w)} {\color{black}]}} \rightarrow ({\color{brown}z \leq w})\}} {\color{OliveGreen} )} {\color{blue}]}\} ~\rangle. 
\end{align*}

Notice that any quantification over a standard set was ``transferred'' to a quantification over the corresponding nonstandard extension of that set. The nonquantified occurrences of $\in$ in the original sentence were as relation symbols (that is, `$a \in b$' is true just in case $a$ is an element of $b$). Strictly speaking, an occurence of a relation (or function) symbol must be transferred to the nonstandard extension of that relation (function) symbol. Thus, the second line of the transferred sentence must technically be `${\color{violet}[\exists x \in {^*}\mathbb{R}} {\color{OliveGreen}(\forall y \in {^*}\mathbb{R}} \{{\color{red}(y ~{^*}\in A) \rightarrow (y ~{^*}\leq x)} \} {\color{OliveGreen})} {\color{violet}]}$'. However, as before, we suppress the ${^*}$ on the transferred relation symbols for better readability. 

In practice, we often write informal logic sentences as long as it is clear that they can me made formal. For instance, instead of writing $$ {\color{OliveGreen}(\forall y \in \mathbb{R}} \{{\color{red}(y \in A) \rightarrow (y \leq x)} \} {\color{OliveGreen})},$$ one would often write `$\forall y \in A {\color{red}(y \leq x)}$'.

The above discussion implies that $\mathbb{N}$ would not be an element of ${^*}\mathcal{P}(\mathbb{R})$ (as it does not satisfy the ${^*}$-least upper bound property), whatever the latter object is (the object ${^*}\mathcal{P}(\mathbb{R})$ would in fact be a subset of $\mathcal{P}({^*}\mathbb{R})$ due to the transfer of the sentence $``\forall A \in \mathcal{P}(\mathbb{R}) ~\forall x \in A ~(x \in \mathbb{R})"$). As we shall see, we do indeed extend $\mathcal{P}(\mathfrak{X})$. An element of ${^*}\mathcal{P}(\mathfrak{X})$ is called an \textit{internal subset} of ${^*}\mathfrak{X}$. The previous example leads to the observation that ${^*}\mathcal{P}(\mathbb{R})$ is not a superset of $\mathcal{P}(\mathbb{R})$ in the literal sense. It does, however, contain as an element the extension ${^*}A$ for any $A \in \mathcal{P}(\mathbb{R})$.

In general, we fix a set $\mathfrak{X}$ consisting of atoms, and extend what is called the superstructure $V(\mathfrak{X})$ of $\mathfrak{X}$, which is defined inductively as follows:

\begin{equation}
    \label{superstructure} 
    \begin{array}{rcl} 
     V_0(\mathfrak{X}) &\defeq &  \mathfrak{X}, \\
     V_{n}(\mathfrak{X}) &\defeq & \mathcal{P}(V_{n-1}(\mathfrak{X})) \text{ for all } n \in \mathbb{N}, \\
     V(\mathfrak{X}) &\defeq &  \bigcup_{n \in \mathbb{N} \cup \{0\}} V_n(\mathfrak{X}).
     \end{array}
\end{equation}

By choosing $\mathfrak{X}$ suitably, the superstructure $V(\mathfrak{X})$ can be made to contain all mathematical objects relevant for a given theory. For example, if $\mathbb{R} \subseteq \mathfrak{X}$, then all collections of subsets of $\mathbb{R}$, including all topologies on $\mathbb{R}$, all sigma-algebras on $\mathbb{R}$, etc., live as objects in $V_2(\mathfrak{X}) \subseteq V(\mathfrak{\mathfrak{X}})$. For a finite subset consisting of $k$ objects from $V_m(\mathfrak{X})$, the ordered $k$-tuple of those objects is an element of $V_n(\mathfrak{X})$ for some larger $n$ (and hence the set of all $k$-tuples of objects in $V_m(\mathfrak{X})$ lies as an object in $V_{n+1}(\mathfrak{X})$). For example, if $x, y \in V_m(\mathfrak{X})$, then the ordered pair $(x, y)$ is just the set $\{\{x\}, \{x,y\}\} \in V_{m+2}(\mathfrak{X})$. Identifying functions and relations with their graphs, $V{(\mathfrak{X})}$ also contains, if $\mathbb{R} \subseteq \mathfrak{X}$, all functions from $\mathbb{R}^n$ to $\mathbb{R}$, all relations on $\mathbb{R}^n$, etc., for all $n \in \mathbb{N}$. 

We extend the superstructure $V(\mathfrak{X})$ via a \textit{nonstandard map},
$${^*}: V(\mathfrak{X}) \rightarrow V({^*}\mathfrak{X}),$$
which, by definition, is any map satisfying the following axioms:
\begin{enumerate}
    \item[(NS1)]\label{NS1} The transfer principle holds.
    \item[(NS2)]\label{NS2} ${^*}\alpha = \alpha$ for all $\alpha \in \mathfrak{X}$.
    \item[(NS3)]\label{NS3} $\{{^*}a: a \in A\} \subsetneq {^*}A$ for any infinite set $A \in V(\mathfrak{X})$.
\end{enumerate}

We refer to Albeverio et al. \cite[Chapter 1]{Albeverio} for discussions on superstructures and the nuances of the two different ways of using the symbol $\in$ (for quantifying over a set, and as a membership relation symbol for a superstructure). A nonstandard map may not be unique. In practice, however, we fix a standard universe $V(\mathfrak{X})$ and a nonstandard map ${^*}$. The reader is referred to Chang--Keisler \cite[Theorem 4.4.5, p. 268]{Model_Theory} or Albeverio et al. \cite[Chapter 1]{Albeverio} for a proof of the existence of a nonstandard map.

An object that belongs to ${^*}A$ for some $A \in V(\mathfrak{X})$ is called \textit{internal}. We have already seen several examples of internal sets and functions: ${^*}\mathbb{N}$, ${^*}\mathbb{R}$, ${^*}f$ (for any standard function $f$), etc. Unlike these examples, (NS3) guarantees the existence of internal objects that are not ${^*}\alpha$ for any $\alpha \in V(\mathfrak{X})$. For instance, for any infinite element $N$ of ${^*}\mathbb{N}$ (such elements exist because of Proposition \ref{infinite natural} which we will study below), the set $\{1, \ldots, N\}$ of the ``first $N$ nonstandard natural numbers'' is internal, yet it does not equal the nonstandard extension of any standard set. The fact that it is internal follows from the transfer of the following standard sentence:

\begin{align*}
    \forall n \in \mathbb{N} ~\exists! A \in \mathcal{P}(\mathbb{N}) ~[\forall x \in \mathbb{N} (x \in A \leftrightarrow x \leq n)].
\end{align*}

Internal objects are those that inherit properties from their standard counterparts by transfer. Thus, for example, the transfer of Archimedean property (see (\ref{*-Archimedean})) says that ${^*}\mathbb{N}$ does not have an upper bound. Note that, by transfer, the class of internal sets is closed under Boolean operations such as finite unions, finite intersections, etc.

\begin{definition}
For a cardinal number $\kappa$, a nonstandard extension is called \textit{$\kappa$-saturated} if any collection of internal sets that has cardinality less than $\kappa$ and that has the finite intersection property has a non-empty intersection.
\end{definition}

 We will henceforth assume that the nonstandard extension we work with is sufficiently saturated (cf. Chang--Keisler \cite[Lemma 5.1.4, p. 294 and Exercise 5.1.21, p. 305]{Model_Theory}). 
 
An element in ${^*}\mathbb{R}$ will be called \textit{infinite} if it is larger than all elements in $\mathbb{R}$. Similarly an element in ${^*}\mathbb{R}$ will be called an \textit{infinitesimal} if its absolute value (that is, its image under the extension of the absolute value map) is smaller than all nonzero elements in $\mathbb{R}$.

\begin{proposition}\label{existence of infinites}
${^*}\mathbb{R}$ contains infinite as well as infinitesimal elements.
\end{proposition} 
\begin{proof}
Any element in the non-empty intersection $\cap_{n \in \mathbb{N}} \{x \in {^*}\mathbb{R}: x>n\}$ must be infinite. The multiplicative inverse of any infinite element is infinitesimal.
\end{proof}

The next result holds since elements of ${^*}\mathbb{N}$ are at least one unit apart (due to transfer) and any finite element of ${^*}\mathbb{N}$ has a least natural number larger than it.

\begin{proposition}\label{infinite natural}
Any $N \in {^*}\mathbb{N} \backslash \mathbb{N}$ is hyperfinite. We express this by writing $N > \mathbb{N}$.
\end{proposition}

The following result is a consequence of the fact that $\mathbb{N}$ is not internal. See also Albeverio et al. \cite[Proposition 1.2.7, p.21]{Albeverio}.

\begin{proposition}\label{Over and under} Let $A$ be an internal set.
\begin{enumerate}[(i)]
    \item\label{overflow}[\textbf{Overflow}] If $\mathbb{N} \subseteq A$, then there is an $N > \mathbb{N}$ such that $$\{n \in {^*}\mathbb{N}: n \leq N\} \subseteq A.$$ 
    \item\label{underflow}[\textbf{Underflow}] If $A$ contains all hyperfinite natural numbers, then there is an $n_0 \in \mathbb{N}$ such that ${^*}\mathbb{N}_{\geq n_0} \defeq \{n \in {^*}\mathbb{N}: n \geq n_0\} \subseteq A$.
\end{enumerate}
\end{proposition}

The next result says that one can think of a finite nonstandard real number $z$ as having a real part, and an infinitesimal part (in fact, this real part is just $\sup\{y \in \mathbb{R}: y \leq z\}$). See Cutland \cite[Theorem 2.10, p. 55]{Cutland_NATO} for a proof.
\begin{proposition}\label{standard part map}
For all $z \in {^*}\mathbb{R}_{\text{fin}}$, there is a unique $x \in \mathbb{R}$ (called the \textit{standard part} of $z$) such that $(z - x)$ is infinitesimal. We write $\st(z) = x$ (or $z \approx x$; ${\degree}z = x$, etc). 
\end{proposition}

Note that, more generally, one can define the notion of standard parts for elements in the nonstandard extension of any Hausdorff space (in general, we will need a point to be \textit{nearstandard}, instead of finite, for it to have a standard part). See Albeverio et al. \cite[p.48]{Albeverio} for that discussion. While we omit such a general setting in this paper, it is useful to know that the notions of finite points and standard parts have natural generalizations to finite-dimensional Euclidean spaces. Thus for $k \in \mathbb{N}$, an element $x \in {^*}\mathbb{R}^k$ is \textit{finite} if and only if $\norm{x} \in {^*}\mathbb{R}_{\text{fin}}$, and for any finite $x \in ({^*}\mathbb{R}^k)$, there is a unique $y \in \mathbb{R}^k$ satisfying $\norm{x - y} \approx 0$, which we call the \textit{standard part} of $x$.  

Using the notion of standard parts, we have the following useful characterization of continuity and uniform continuity (see, for example, Albeverio et al. \cite[Proposition 1.3.3, p.27]{Albeverio} for the one-dimensional case, with the higher dimensional case following a similar argument):

\begin{proposition}
Let $k , \ell \in \mathbb{N}$ and $f: \mathbb{R}^k \rightarrow \mathbb{R}^\ell$ be a function. Then:
\begin{enumerate}
    \item $f$ is continuous at $x \in \mathbb{R}^k$ if and only if ${^*}f(\st^{-1}(x)) \subseteq \st^{-1}(f(x))$.
    \item $f$ is uniformly continuous if and only if for any $x, y \in {^*}\mathbb{R}^k$ with $\norm{x - y} \approx 0$, we have $\norm{{^*}f(x) - {^*}f(y)} \approx 0$ as well.
\end{enumerate}
\end{proposition}

The next result gives a nice characterization of limit points of sequences (see Cutland \cite[Theorems 3.1 and 3.3]{Cutland_NATO} for proofs of the two statements):

\begin{proposition}\label{limits of sequences}
For a sequence of real numbers $\{a_n\}_{n \in \mathbb{N}}$, there is an extended sequence $\{a_n\}_{n \in {^*}\mathbb{N}}$ (by viewing the original sequence as a function on $\mathbb{N}$). 
A real number $L$ is an accumulation point of the sequence $\{a_n\}_{n \in \mathbb{N}} \iff \text{there is an } N > \mathbb{N}$ such that $\st(a_N) = L$. Thus $\lim a_n = L \iff \st(a_N) = L$ for all $N > \mathbb{N}$. 
\end{proposition}

The following consequence of saturation will be useful in the sequel. Since this result is very important, we will provide a proof (see also Albeverio et al. \cite[Lemma 3.1.1, p. 64]{Albeverio}).

\begin{proposition}\label{countable union}
A countable union of disjoint internal sets is internal if and only if all but finitely many of them are empty.
\end{proposition}
\begin{proof}
Suppose $\{A_i\}_{i \in \mathbb{N}}$ is a countable collection of disjoint internal sets. Let $A = \cup_{i \in \mathbb{N}}A_i$. If all but finitely many of the $A_i$ are empty, then $A$ being a finite union of internal sets is also internal due to transfer.

Conversely, if $A$ is internal, then $A \backslash A_i$ is internal for each $i \in \mathbb{N}$ by transfer. In that case, if all but finitely many of the $A_i$ are not empty, then the collection $\{A \backslash A_i\}_{i \in \mathbb{N}}$ would satisfy the finite intersection property. By saturation, this would lead to $\cap_{i \in \mathbb{N}} (A \backslash A_i) \neq \emptyset$, which is absurd. This completes the proof by contradiction.
\end{proof}

\subsection{Loeb Measures}

Let $\Omega$ be an internal set in a nonstandard universe ${^*}V(\mathfrak{X})$. Let $\mathcal{F}$ be an \textit{internal algebra} on $\Omega$, that is, an internal set consisting of subsets of $\Omega$ that is closed under complements and finite unions. Given a finite, finitely additive internal measure $\mathbb{P}$ (that is, $\mathbb{P}: \mathcal{F} \rightarrow {^*}\mathbb{R}_{\geq 0}$ satisfies $\mathbb{P}(\emptyset) = 0$, $\mathbb{P}(\Omega) < \infty$, and $\mathbb{P}(A \cup B) = \mathbb{P}(A) + \mathbb{P}(B)$ whenever $A \cap B = \emptyset$), the map $\st(\mathbb{P}) : \mathcal{F} \rightarrow {\mathbb{R}_{\geq 0}}$ is an ordinary finite, finitely additive measure. By Proposition \ref{countable union}, it follows that $\st(\mathbb{P})$ satisfies the premises of Carath\'eodory Extension Theorem. By that theorem, it extends to a unique measure on $\sigma(\mathcal{F})$ (the smallest sigma algebra containing $\mathcal{F}$), whose completion is called the \textbf{Loeb measure} of $\mathbb{P}$. The corresponding complete measure space $(\Omega, L(\mathcal{F}), L\mathbb{P})$ is called the \textbf{Loeb space} of $(\Omega, \mathcal{F}, \mathbb{P})$. 

It is in general difficult to visualize Loeb measurable sets that are not in the original internal algebra. The following approximation result helps us to approximate any Loeb measurable sets by sets in the original internal algebra. See Albeverio et al. \cite[Theorem 3.1.2, p. 64]{Albeverio} for a proof.

\begin{proposition}\label{approximating by *Borel sets}
Let $(\Omega, L(\mathcal{F}), L\mathbb{P})$ be the Loeb probability space of $(\Omega, \mathcal{F}, \mathbb{P})$. 
\begin{enumerate}[(i)]
    \item\label{approximate 1} For each $A \in \sigma(\mathcal{F})$, there is a set $B \in \mathcal{F}$ such that $L\mathbb{P}(A \triangle B) = 0$.
    
    \item\label{approximate 2} For each $A \in \sigma(\mathcal{F})$ and $\epsilon \in \mathbb{R}_{>0}$, there are sets $B, D \in \mathcal{F}$ such that $B \subseteq A \subseteq D$ and
    $$L\mathbb{P}(D) - \epsilon \leq L\mathbb{P}(A) \leq L\mathbb{P}(B) + \epsilon.$$
\end{enumerate}
\end{proposition}

We will use the following simplification of Ross \cite[Theorem 5.1, p. 105]{Ross_NATO} extensively:
\begin{proposition}\label{Loeb measurability}
Let $(\Omega, L(\mathcal{F}), L\mathbb{P})$ be the Loeb probability space of $(\Omega, \mathcal{F}, \mathbb{P})$. Suppose $F: \Omega \rightarrow {^*}\mathbb{R}$ is an internal function that is measurable in the sense that $F^{-1}({^*}B) \in \mathcal{F}$ for all $B \in \mathcal{B}(\mathbb{R})$ (where $\mathcal{B}(\mathbb{R})$ is the Borel $\sigma$-algebra on $\mathbb{R}$). If $F(\omega) \in {^*}\mathbb{R}_{\text{fin}}$ for $L\mathbb{P}$-almost all $\omega \in \Omega$, then $\st(F)$ is Loeb measurable (that is, measurable as a map from $(\Omega, L(\mathcal{F}))$ to $(\mathbb{R}, \mathcal{B}(\mathbb{R}))$).
\end{proposition}

For a standard measure space $(\Omega, \mathcal{F})$, let $\Prob(\Omega, \mathcal{F})$ be the set of probability measures on $(\Omega, \mathcal{F})$. If $\mathcal{C} \in V(\mathfrak{X})$ is a collection of measure spaces, then $\Prob(\mathcal{C})$ denotes the set of all probability measures on elements in $\mathcal{C}$. Any element in ${^*}\Prob(\mathcal{C})$ is a finitely additive internal probability on an internal measure space. For any $\mathbb{P} \in \Prob(\mathcal{C})$, there is an $\textit{integral operator}$ that takes certain functions (those in the space $L^1(\mathbb{P})$ of integrable real-valued functions on the underlying sample space of $\mathbb{P}$) to their integrals with respect to $\mathbb{P}$. Thus if $(\Omega, \mathcal{F}, \mathbb{P}) \in {^*}\mathcal{C}$ is an internal probability space, we also have the associated space ${^*}L^1(\Omega, \mathbb{P})$ of ${^*}$-integrable functions. 

For any  ${^*}$-integrable $F: \Omega \rightarrow {^*}\mathbb{R}$, one then has $\starint_\Omega F d\mathbb{P} \in {^*}\mathbb{R}$, which we call the ${^*}$-\textit{integral} of $F$ over $(\Omega, \mathbb{P})$. This ${^*}$-integral on ${^*}L^1(\Omega)$ inherits many properties (an important one being linearity) from the ordinary integral by transfer. If $F$ is finite almost surely with respect to the corresponding Loeb measure, then $\st(F)$ is Loeb measurable by Proposition \ref{Loeb measurability}. In that case, it is interesting to study the relation between the ${^*}$-integral of $F$ and the Loeb integral of $\st({^*}F)$. The following result covers this for a useful class of functions (see Ross \cite[Theorem 6.2, p.110]{Ross_NATO} for a proof):
\begin{theorem}\label{S-integrable TFAE}
Suppose $(\Omega, \mathcal{F}, \mathbb{P})$ is an internal probability space and $F \in {^*}L^1(\Omega)$ is such that $L\mathbb{P}(F \in {^*}\mathbb{R}_{\text{fin}}) = 1$. Then the following are equivalent:
\begin{enumerate}[(1)]
    \item\label{S1} $\starint_\Omega \abs{F} d\mathbb{P} \in {^*}\mathbb{R}_{\text{fin}}$, and 
    $$\st\left(\starint_\Omega \abs{F} d\mathbb{P} \right) = \lim_{m \rightarrow \infty} \st\left(\starint_\Omega  \abs{F}\mathbbm{1}_{\{\abs{F} \leq m\}} \right).$$
    
    \item\label{S2} For every $M > \mathbb{N}$, we have $\st\left(\starint_\Omega \abs{F} \mathbbm{1}_{\{\abs{F} > M\}}d \mathbb{P}\right) = 0$.
    \item\label{S3} $\starint_\Omega \abs{F} d\mathbb{P} \in {^*}\mathbb{R}_{\text{fin}}$; and for any $B \in \mathcal{F}$ we have: 
    $$\mathbb{P}(B) \approx 0 \Rightarrow \starint_\Omega \abs{F} \mathbbm{1}_B d\mathbb{P} \approx 0.$$
    \item\label{S4} $\st(F)$ is Loeb integrable, and $\st\left(\starint_\Omega \abs{F} d\mathbb{P} \right) = \int_\Omega \abs{\st(F)} dL\mathbb{P}$.
\end{enumerate}
\end{theorem}

A function satisfying the conditions in Theorem $\ref{S-integrable TFAE}$ is called $\mathit{S}$\textit{-integrable} on $(\Omega, \mathcal{F}, \mathbb{P})$. The notion of $S$-integrability, first developed by Anderson \cite{Anderson-1976}, is one of the most ubiquitous concepts in nonstandard measure theory. Given a Loeb measurable $f: \Omega \rightarrow \mathbb{R}$, a natural question is when does it occur as the standard part of an internal function. An internal measurable function $F: \Omega \rightarrow {^*}\mathbb{R}$ is called a \textit{lifting} of a Loeb measurable function $f$ if $L\mathbb{P}(\st(F) = f) = 1$. 

The following theorem shows that ${^*}$-integrable functions can be characterized as those possessing $S$-integrable liftings (see Ross \cite[Theorem 6.4, p.111]{Ross_NATO} for a proof).

\begin{theorem}\label{S-integrable lifting}
Let $(\Omega, \mathcal{F}, \mathbb{P})$ be an internal probability space and let $(\Omega, L(\mathcal{F}), L(\mathbb{P}))$ be the associated Loeb space. Suppose $f: \Omega \rightarrow \mathbb{R}$ is Loeb measurable. Then $f$ is Loeb integrable if and only if it has an $S$-integrable lifting.
\end{theorem}

We finish our review of basic nonstandard methods with the following remark about the nature of the standard universe we are extending in this paper. 

\begin{remark}\label{superstructure remark}
Let $\mathfrak{X}$ be a set of urelements and let $V(\mathfrak{X})$ be its superstructure. As discussed earlier, we fix a sufficiently saturated nonstandard extension of $V(\mathfrak{X})$. In this paper, we work with measures defined on a sequence of measure spaces, and want to construct a natural Loeb measure on any element in the nonstandard extension of such a sequence. One issue in doing so could be that the measure spaces might not all lie in a single iterated power set over $\mathfrak{X}$ (in which case, we cannot think of the sequence of measure spaces as an element of $V(\mathfrak{X})$). In particular, this would be an issue if our measure spaces were the Borel spaces $(\mathbb{R}^n, \mathcal{B}(\mathbb{R}^n))$ and $\mathfrak{X}$ was the set of real numbers. To get around this difficulty, we take a set $\mathfrak{X}$ that contains (copies of) $\mathbb{R}^n$ for each $n \in \mathbb{N}$.
\end{remark}
\end{section}

\section{A Quick nonstandard proof of Poincar\'e's theorem}\label{section 2}
Using the nonstandard characterization of limit points, Poincar\'e's theorem is essentially a statement about the Loeb measure of the fiber (in the hyperfinite-dimensional sphere $S^{N-1}(\sqrt{N})$ for $N > \mathbb{N}$) of a finite-dimensional set equaling its Gaussian measure. In a more general setting, we analyze this type of phenomenon in the next subsection. These results are routine but essential in setting up later proofs. 

\subsection{When a Loeb measure matches up with a standard measure on a subspace}
In what follows, there will be a measure space $(E, \mathcal{E})$ such that we assume $\mathfrak{X}$ to contain copies of $E^n$ for all $n \in \mathbb{N}$. The corresponding product sigma-algebra on $E^n$ will be denoted by $\mathcal{E}_n$. Recall that we will be working with a sufficiently saturated nonstandard extension of the superstructure $V(\mathfrak{X})$ over $\mathfrak{X}$. Let $k \in \mathbb{N}$. For $n \in \mathbb{N}_{\geq k}$, if $\Omega \in \mathcal{E}_n$ and $\nu$ is a measure on the induced sub-sigma-algebra on $\Omega$, then for any $B \in \mathcal{E}_k$, we denote $\nu(\Omega \cap (B \times E^{n - k}))$ by $\nu(B)$. Similarly, we can talk about integrating a measurable function $f\co E^k \to \mathbb{R}$ over $\Omega$ by extending $f$ canonically to $E^n$.

\begin{proposition}\label{Borel equivalence}
Let $\Omega \in {^*}V(\mathfrak{X})$ be such that $\Omega \subseteq {^*}E^N$ for some $N \in {^*}\mathbb{N}$. Let $\mathcal{E}$ be a sigma-algebra on $E$, and let $\mathcal{E}_k$ denote the corresponding product sigma-algebra on $E^k$ for each $k \in \mathbb{N}$. Let ${^*}\mathcal{E}_N$ denote the corresponding internal algebra on ${^*}E^N$ (defined by extension of the sequence $\{\mathcal{E}_k\}_{k \in \mathbb{N}}$, which is an element of $V(\mathfrak{X})$ when viewed as a function on $\mathbb{N}$). Let $\mathcal{F}$ be the restriction of ${^*}\mathcal{E}_N$ to $\Omega$. 

Fix $k \in \mathbb{N}$ and suppose $\mathbb{P} \in \Prob(E^k, \mathcal{E}_k)$. Let $\nu \in {^*}\Prob(\Omega, \mathcal{\mathcal{F}})$. If $L\nu$ is the corresponding Loeb measure, and if $N \geq k$, then:
\begin{gather}
    \int_{\Omega} \st({^*}f) dL\nu = \int_{E^k} f d\mathbb{P} \text{ for all bounded measurable } f\co E^k \rightarrow \mathbb{R} \label{Prop 1}\\
     \Updownarrow \nonumber \\
    L\nu ({^*}B) = \mathbb{P}(B) \text{ for all } B \in \mathcal{E}_k. \label{prop 1'}
\end{gather}
\end{proposition}

\begin{proof}
If $f \co E^k \to \mathbb{R}$ is bounded measurable, then $\st({^*}f)$ is Loeb measurable on $\Omega$ by Proposition \ref{Loeb measurability}. Hence the left side of equation \eqref{Prop 1} is well-defined. 

The forward implication is immediate by taking $f = \mathbbm{1}_B$, the indicator function of $B \in \mathcal{E}_k$. For the reverse implication, assume that $L\nu ({^*}B) = \mathbb{P}(B) \text{ for all } B \in \mathcal{E}_k$ (that is, indicator functions of measurable sets satisfy (\ref{Prop 1})). The set of functions satisfying (\ref{Prop 1}) is closed under taking finite $\mathbb{R}$-linear combinations, and hence all simple functions satisfy (\ref{Prop 1}). Fix a bounded measurable function $f\co E^k \to \mathbb{R}$. By standard measure theory (see, for example, Folland \cite[Theorem 2.10]{Folland}), there is a sequence $\{f_n\}_{n \in \mathbb{N}}$ of simple functions that converges to $f$ uniformly on $E^k$.

For $\epsilon \in \mathbb{R}_{>0}$, find $n_\epsilon \in \mathbb{N}$ such that we have the following inequality. $$\abs{f_n(x) - f(x)} < \epsilon \text{ for all } x \in E^k \text{ and } n \in \mathbb{N}_{\geq n_{\epsilon}}.$$ 
By transfer, for all $n \in \mathbb{N}_{\geq n_{\epsilon}}$, we get $\abs{{^*}f_n(x) - {^*}f(x)} < \epsilon$ on ${{^*}E}^k$. Hence, $$\abs{\st({^*}f_n(x)) - \st({^*}f(x))} \leq \epsilon \text{ for all } n \in \mathbb{N}_{\geq n_{\epsilon}} \text{ and } x \in {^*}E^k.$$ 

As a consequence, we get: 
\begin{align*}
    &\abs{\int_{\Omega} \st({^*}f) dL\nu - \int_{\Omega} \st({^*}f_n) dL\nu } \leq \epsilon \text{ for all } n \in \mathbb{N}_{\geq n_{\epsilon}}, \\
     \text{that is, } &\abs{\int_{\Omega} \st({^*}f) dL\nu - \int_{E^k} f_n d\mathbb{P}} \leq \epsilon \text{ for all } n \in \mathbb{N}_{\geq n_{\epsilon}}.
\end{align*}
But $\lim_{n \rightarrow \infty} \int_{E^k} f_n d\mathbb{P} = \int_{E^k} f d\mathbb{P}$, by dominated convergence theorem. Since $\epsilon \in \mathbb{R}_{>0}$ is arbitrary, this implies $\int_{\Omega} \st({^*}f) dL\nu = \int_{E^k} f d\mathbb{P}$, completing the proof.
\end{proof}

The hypothesis in Proposition \ref{Borel equivalence} is an abstract rendering of the premise of our central problem about limits of spherical measures. Indeed, we may think of $E$ as $\mathbb{R}$, the space $\Omega$ as the hyperfinite dimensional sphere $S^{N-1}(\sqrt{N})$ for some $N > \mathbb{N}$, and $\mathbb{P}$ as the standard Gaussian measure $\mu$. Then, \eqref{prop 1'} is the nonstandard characterization of \eqref{Poincare's first limit}, while \eqref{Prop 1} corresponds to \eqref{Poincare limit}. To strengthen this theme, in the next subsection, we will take a standard sequence of probability spaces and replace $\Omega$ by the $N^{\text{th}}$ term (for any $N > \mathbb{N}$) of the nonstandard extension of that sequence. We first record some useful implications of Proposition \ref{Borel equivalence} below. 

\begin{corollary}\label{functions are finite almost everywhere}
In the setting of Proposition \ref{Borel equivalence}, suppose \eqref{Prop 1}, and hence \eqref{prop 1'}, hold. Then $$L\nu(\{x \in \Omega: {^*}f(x) \in {^*}\mathbb{R}_{\text{fin}}\}) = 1 \text{ for all measurable } f\co E^k \to \mathbb{R}.$$ 
\end{corollary}
\begin{proof}
If $B_n \defeq \{x \in E^k: \abs{f(x)} < n\}$ for $n \in \mathbb{N}$, then the required probability is 
\begin{align*}
  L\nu \left(\cup_{n \in \mathbb{N}} {^*}B_n \right) &= \lim_{n \rightarrow \infty} L\nu ({^*}B_n) 
    \overset{\eqref{prop 1'}}{=} \lim_{n \rightarrow \infty} \mathbb{P}(B_n) 
    = 1, &~ 
\end{align*}
thus completing the proof. 
\end{proof}

\begin{corollary}\label{from bounded to integrable}
In the setting of Proposition \ref{Borel equivalence}, suppose \eqref{Prop 1} holds. Then, for any $\mathbb{P}$--integrable function $f \co E^k \to \mathbb{R}$, we have that $\st({^*}f)$ is $L\nu$--integrable; and furthermore, 

$$\int_{\Omega} \abs{\st({^*}f)} dL\nu = \int_{E^k} \abs{f} d\mathbb{P} \text{, and } \int_{\Omega} \st({^*}f) dL\nu = \int_{E^k} f d\mathbb{P}.$$
\end{corollary}
\begin{proof}
We see that $\st({^*}f)$ is Loeb measurable on $\Omega$ by Corollary \ref{functions are finite almost everywhere} and Proposition \ref{Loeb measurability}. Also, by Corollary \ref{functions are finite almost everywhere}, $\st({^*}f) \mathbbm{1}_{\{\abs{{^*}f} < n\}} \uparrow \st({^*}f)$ $L\nu$--almost surely. Hence, we have: 
\begin{align*}
    \int_{\Omega} \abs{\st({^*}f)} dL\nu &= \lim_{n \rightarrow \infty} \int_{\Omega} \st({^*}\abs{f}) \cdot \mathbbm{1}_{\{{^*}\abs{f} \leq n\}} dL\nu  \\
    &= \lim_{n \rightarrow \infty} \int_{E^k} \abs{f} \cdot \mathbbm{1}_{\{\abs{f} \leq n\}}  d\mathbb{P}\\
    &= \int_{E^k} \abs{f} d\mathbb{P} < \infty.  
\end{align*}

The first line follows from the monotone convergence theorem (applied on the Loeb space $(\Omega, L(\mathcal{F}), L\nu)$), the second line follows from \eqref{Prop 1}, and the third line follows from the monotone convergence theorem (applied on the probability space $(E^k, \mathcal{E}_k, \mathbb{P})$). 

Now, since $\lim_{n \rightarrow \infty} \left(\st({^*}f) \cdot \mathbbm{1}_{\{\abs{{^*}f} < n\}}\right) = \st({^*}f) ~~L\nu$-almost surely (using Corollary \ref{functions are finite almost everywhere}), and since $\abs{\st({^*}f) \cdot \mathbbm{1}_{\{\abs{{^*}f} < n\}}} \leq \abs{\st({^*}f)} \in L^1(\Omega, L\nu)$, it follows that:
\begin{align*}
     \int_{\Omega} {\st({^*}f)} dL\nu &= \lim_{n \rightarrow \infty} \int_{\Omega} \st({{^*}f}) \cdot \mathbbm{1}_{\{\abs{{^*}f} \leq n\}} dL\nu \\
     &= \lim_{n \rightarrow \infty} \int_{E^k} {f} \cdot \mathbbm{1}_{\{\abs{f} \leq n\}}  d\mathbb{P} \\
     &= \int_{E^k} {f} d\mathbb{P}.
\end{align*}

The first line follows from the dominated convergence theorem (applied on the Loeb space $(\Omega, L(\mathcal{F}), L\nu)$), the second line follows from \eqref{Prop 1}, and the third line follows from the dominated convergence theorem (applied on the measure space $(E^k, \mathcal{E}_k, \mathbb{P})$). This completes the proof.
\end{proof}

\begin{corollary}\label{TFAE'}
In the setting of Proposition \ref{Borel equivalence}, the following are equivalent:

\begin{enumerate}[(1)]
    \item\label{itm1'} $\int_{\Omega} \st({^*}f) dL\nu = \int_{E^k} f d\mathbb{P} \text{ for all bounded measurable } f\co E^k \to \mathbb{R}$.    
    \item\label{itm2'} $L\nu ({^*}B) = \mathbb{P}(B) \text{ for all } B \in \mathcal{E}_k$.
    \item\label{itm3'} $L\nu({^*}B) \leq \mathbb{P}(B)$ for all $B \in \mathcal{E}_k$.
    \item\label{itm4'} $L\nu({^*}B) \geq \mathbb{P}(B)$ for all $B \in \mathcal{E}_k$.
\end{enumerate}
\end{corollary}

\begin{proof}
\ref{itm1'} $\Leftrightarrow$ \ref{itm2'} follows from Proposition \ref{Borel equivalence}. Also, \ref{itm3'} and \ref{itm4'} follow from \ref{itm2'} immediately. Conversely, assume \ref{itm3'}. For any Borel set $B \subseteq E^k$, we have 
\begin{align}
L\nu ({^*}B) &\leq \mathbb{P}(B) \text{, and } \label{O}\\
L\nu ({^*}E^k \backslash {^*}B) \leq \mathbb{P}(E^k \backslash B) \Rightarrow L\nu({^*}B) &\geq \mathbb{P}(B). \label{OC}
\end{align}
Combining \eqref{O} and \eqref{OC} gives \ref{itm2}. The proof of \ref{itm4'} $\Rightarrow$ \ref{itm2'} is similar. 
\end{proof}

We end this subsection with the remark that if $E$ is a Hausdorff topological space equipped with its Borel sigma-algebra, and if the probability measure $\mathbb{P}$ is Radon, then \eqref{Prop 1} and \eqref{prop 1'} are both equivalent to the Loeb measure $L\nu$ agreeing with $\mathbb{P}$ on the nonstandard extensions of all open (or all compact) subsets of $E$.

\begin{proposition}\label{TFAE}
In the setting of Proposition \ref{Borel equivalence}, suppose $E$ is a Hausdorff topological space and let $\mathcal{B}(E^k)$ be the Borel sigma-algebra on $E^k$. If $\mathbb{P}$ is a Radon probability measure on $E^k$, then the following are equivalent:

\begin{enumerate}[(1)]
    \item\label{itm1} $\int_{\Omega} \st({^*}f) dL\nu = \int_{E^k} f d\mathbb{P} \text{ for all bounded Borel measurable } f\co E^k \to \mathbb{R}$.    
    \item\label{itm2} $L\nu ({^*}B) = \mathbb{P}(B) \text{ for all } B \in \mathcal{B}(E^k)$.
     \item\label{itm3} $L\nu({^*}B) \leq \mathbb{P}(B)$ for all $B \in \mathcal{B}(E^k)$.
     \item\label{itm4} $L\nu({^*}B) \geq \mathbb{P}(B)$ for all $B \in \mathcal{B}(E^k)$.
    \item\label{itm5} $L\nu ({^*}O) = \mathbb{P}(O) \text{ for all open sets } O \subseteq {E}^k$.    
    \item\label{itm6} $L\nu ({^*}C) = \mathbb{P}(C) \text{ for all compact sets } C \subseteq {E}^k$.
\end{enumerate}
\end{proposition}

\begin{proof}
The equivalence of \ref{itm1}, \ref{itm2}, \ref{itm3}, and \ref{itm4} has been established without any conditions on $\mathbb{P}$ in the previous corollary. Also, \ref{itm2} $\Rightarrow$ \ref{itm5} is immediate. To complete the proof, we will show that \ref{itm5} $\Rightarrow$ \ref{itm6} and \ref{itm6} $\Rightarrow$ \ref{itm4}. 

To see \ref{itm5} $\Rightarrow$ \ref{itm6}, note that if $C$ is a compact subset of the Hausdorff space $E^k$, then $C$ is closed, so that the subset $O \defeq E^k \backslash C$ is open. By using the fact that ${^*}C = {^*}E^k \backslash {^*}O$, and then applying \ref{itm5} to $O$, we obtain the following:
\begin{align*}
    L\nu({^*}C) = 1 - L\nu({^*}O) &= 1 - \mbp(O) = \mbp(C).
\end{align*}

We now prove \ref{itm6} $\Rightarrow$ \ref{itm4}. To that end, take any $B \in \mathcal{B}(E^k)$. For any compact subset $C \subseteq B$, we have ${^*}C \subseteq {^*}B$, so that \ref{itm6} implies the following:
\begin{align*}
    L\nu({^*}B) \geq L\nu({^*}C) = \mbp(C) \text{ for all compact subsets $C$ of $B$}.
\end{align*}
Taking supremum over all compact subsets of $B$ and using the fact that the measure $\mbp$ is Radon, we thus obtain the desired inequality as follows:
\begin{align*}
    L\nu({^*}B) \geq \sup \{\mbp(C): C \text{ is a compact subset of $B$}\} = \mbp(B).
\end{align*}
\end{proof}

\subsection{Basic facts about surface area measures and their nonstandard counterparts}
In this subsection, we review three different ways to think about the uniform surface area measure on spheres in Euclidean spaces. One aim of our review is to explain the corresponding internal probability measures on hyperfinite dimensional spheres that we obtain by transfer. We refer to Matilla \cite[Chapter 3]{spherical_measures} and Sengupta \cite[Section 4]{Sengupta} for basic properties of spherical surface area measures.

For each $n \in \mathbb{N}$, we let $\mathcal{B}_n = \mathcal{B}(\mathbb{R}^n)$, the Borel sigma-algebra on $\mathbb{R}^n$, and $O(n)$ be the set of all orthogonal linear transformations of $\mathbb{R}^n$. Let $\mathcal{S}_0$ be the set of all spheres centered at the origin (in any dimension $n \in \mathbb{N}$ and of any radius $r \in \mathbb{R}_{>0}$). Since $\mathfrak{X}$ contains copies of all Euclidean spaces, $\mathcal{S}_0$ is an element of $V(\mathfrak{X})$. Consider the function $\dim \co \mathcal{S}_0 \to \mathbb{N}$ that takes each sphere $S$ to the smallest dimension $n \in \mathbb{N}$ such that $S \subseteq \mathbb{R}^n$. We are being pedantic about the ``smallest dimension'' since we have been identifying (during discussions on measures of sets) a subset $S$ of $\mathbb{R}^n$ with the subset $S \times \mathbb{R}^{n' - n} \subseteq \mathbb{R}^{n'}$ for $n' \in \mathbb{N}_{> n'}$. 

It is known that there is a unique rotation-preserving probability measure on any sphere centered at origin equipped with its Borel sigma-algebra. More formally:
\begin{gather}\nonumber
    \forall S \in \mathcal{S}_0 ~\exists! \bar{\sigma} \in \Prob(S, \mathcal{B}({S})) ~\forall n \in \mathbb{N} \\
    (n = \dim(S)) \rightarrow  ~(\forall R \in O(n) ~\forall A \in \mathcal{B}(S) ~[\bar{\sigma}(R(A)) = \bar{\sigma}(A)]). \label{sigma bar}
\end{gather}

For any $S \in {^*}\mathcal{S}_0$ in the nonstandard universe, the transfer principle implies that the set ${^*}\text{Prob}(S, {^*}\mathcal{B}(S))$ consists of a unique finitely additive internal function, say $\bar{\sigma}_S\co  {^*}\mathcal{B}(S) \to {^*}[0,1]$, that is ${^*}$--rotation preserving and satisfying $\bar{\sigma}_S(S) = 1$. By the usual Loeb measure construction, we get $L\bar{\sigma}_S$ on $\mathcal{L}({^*}\mathcal{B}(S))$ (a sigma-algebra containing $\sigma({^*}\mathcal{B}(\mathcal{S}))$, which we call the \textit{uniform Loeb surface measure} on $S$. As before, we will often drop the superscript $S$ in $\bar{\sigma}_S$ when the sphere is clear from context. 

In finite dimensions, we also have the notion of surface area. For the sphere $S \defeq S^d(R)$ of radius $R \in \mathbb{R}_{>0}$, centered at the origin in $\mathbb{R}^{d+1}$, one can consider the surface area map $\sigma_S \co \mathcal{B}(S) \to \mathbb{R}$, which satisfies the following volume-of-cone formula:

$$\lambda_{d+1}\left(\cup_{0 \leq t \leq 1} tA \right) = \frac{1}{d+1} R \sigma_S(A),$$

where $\lambda_{d+1}$ is the Lebesgue measure on $\mathbb{R}^{d+1}$, and $A \in \mathcal{B}(S)$. This surface area function has the following properties:
\begin{itemize}
\item For any $d \in \mathbb{N}$ and any $R \in \mathbb{R}_{>0}$, we have
        $\sigma_{S^d(R)}(S^d(R)) = c_d \cdot R^{d}$, \text{ where } $c_d = \sigma_{S^d(1)}(S^d(1)) = (d+1) \cdot \frac{\pi^{\frac{d+1}{2}}}{\Gamma\left({\frac{d+1}{2}} + 1\right)} = 2\frac{\pi^{\frac{d+1}{2}}}{\Gamma\left({\frac{d+1}{2}}\right)}$.
\item For any $S \in \mathcal{S}$ and any $A \in \mathcal{B}(\mathcal{S})$, we have $\bar{\sigma}_S(A) = \frac{\sigma_S(A)}{\sigma_S(S)}$, where $\bar{\sigma}_S$ is the rotation preserving probability measure on $S$, as in \eqref{sigma bar}. 
\end{itemize}    

By transfer, we have the notion of ${^*}$-surface area (that is applicable to hyperfinite-dimensional spheres as well) in the nonstandard universe. This could be used as an alternative way to define the uniform Loeb surface measure. 

Yet another way to arrive at the uniform surface area measure on a sphere is by looking at an appropriate pushforward of a Gaussian measure. If $\mu$ is the standard Gaussian measure on $\mathbb{R}^n$ (here $n \in \mathbb{N}$), and $S^{n-1}$ is the unit sphere in $\mathbb{R}^n$, then the rotation invariance of $\mu$ implies that $\mu \circ g^{-1}$ is a rotation invariant probability measure on $S^{n-1}$ (and hence is the same as $\bar{\sigma}$), where   
$$g : \mathbb{R}^n \backslash \{0\} \to S^{n-1} \text{ defined by } g(x) = \frac{x}{\norm{x}}.$$
For spheres centered at origin but having radius $R \in \mathbb{R}_{>0}$, we can use the pushforward through the map $Rg$ (this is scalar multiple by $R$). For instance, for $N > \mathbb{N}$, if $\bar{\sigma}$ is the internal uniform surface area measure on $S^{N-1}(\sqrt{N})$ and $\mu_{(N)}$ is the internal Gaussian measure on ${^*}\mathbb{R}^N$ with mean $\mathbf{0}$ and covariance identity, then for any set $B \in {^*}\mathcal{B}(S^{N-1}(\sqrt{N}))$, we have:
\begin{align}\label{pushforward characterization}
    \bar{\sigma}(B) = \mu_{(N)} \left(\left\{x \in {^*}\mathbb{R}^N : \frac{\sqrt{N}x}{\norm{x}} \in B\right\} \right).
\end{align}

This characterization of the uniform surface area measure yields the classical result of Poincar\'e (Theorem \ref{Poincare's theorem}) without doing any computations. We show that in the next subsection. 

\subsection{A nonstandard proof of Poincar\'e's theorem}\label{nonstandard Poincare}
Suppose $(\Omega, \mathcal{F}, \mbp)$ is a probability space, and $(X_n)_{n \in \mathbb{N}}$ is a sequence of iid $\mathcal{N}(0,1)$ random variables (that is, the $X_i$ are independent Gaussian random variables with mean $0$ and variance $1$). In that case, $({X_n}^2 - 1)_{n \in \mathbb{N}}$ is an iid sequence of random variables with mean zero and finite variance (in fact, the variance is equal to one). Hence the weak law of large numbers implies the following:
\begin{align}\label{WLLN0}
    \lim_{n \rightarrow \infty} \mbp \left( \abs{\frac{({X_1}^2 - 1) + \ldots + ({X_n}^2 - 1)}{n}} > \epsilon\right) = 0 \text{ for all } \epsilon \in \mathbb{R}_{>0}.
\end{align}

Each $X_i$ (where $i \in \mathbb{N}$), as a function from $\Omega$ to $\mathbb{R}$, has a nonstandard extension ${^*}X_i$, which, by transfer, is a ${^*}\mathcal{N}(0,1)$ random variable, that is, ${^*}\mathbb{P} \circ {{^*}X_i}^{-1}$ is the same as the internal measure ${^*}\mu_{(1)}$ (the nonstandard extension of the standard Gaussian measure $\mu_{(1)}$ on $\mathbb{R}$). 

Consider the function $X \co \mathbb{N} \times \Omega \to \mathbb{R}$ defined by:
\begin{align}\label{definition of X}
    X(n, \omega) \defeq X_n(\omega) \text{ for all } n \in \mathbb{N}, \text{ and } \omega \in \Omega. 
\end{align}

Considering the nonstandard extension of $X$, we see that $${^*}X(i, \omega) = {^*}X_i(\omega) \text{ for all } i \in \mathbb{N}, \text{ and } \omega \in \Omega.$$

Furthermore, this allows us to naturally talk about the $N^{\text{th}}$ element of the original sequence of random variables for any $N \in {^*}\mathbb{N}$ (and all those elements will be independent and internally Gaussian distributed with mean $0$ and variance $1$). In the sequel, we will often be loose with notation, and use $X_i$ as both a standard and a nonstandard random variable (when it is considered as a nonstandard random variable, it is understood to be given by the nonstandard extension of the map $X \co \mathbb{N} \times \Omega \to \mathbb{R}$), with the usage being clear from context. 

For the rest of this section, fix $N > \mathbb{N}$. Let $\bar{\sigma}$ be the internal uniform surface area measure on $S^{N-1}(\sqrt{N})$. Let $Y = (X_1)^2 + \ldots + (X_N)^2$. 

\begin{lemma}\label{N over theta lemma}
There exists an infinitesimal $\xi >0$ such that \begin{align}\label{N over theta equation}
    {^*}\mbp \left( \abs{\frac{Y}{N} - 1} > \xi\right) \approx 0.
\end{align}
\end{lemma}

\begin{proof}
Consider $\epsilon \in \mathbb{R}$ such that $0 < \epsilon < 1$. Then we have the following:
\begin{align*}
    {^*}\mbp \left( \abs{\frac{Y}{N} - 1} > \epsilon \right) &= {^*}\mbp \left( \abs{\frac{({X_1}^2 - 1) + \ldots + ({X_N}^2 - 1)}{N}} > \epsilon\right), 
\end{align*}

where the right side is infinitesimal by the nonstandard characterization of limits applied to \eqref{WLLN0}. The lemma now follows by underflow applied to the following internal set.
$$\left\{\epsilon \in {^*}\mathbb{R}_{>0}: {^*}\mbp \left( \abs{\frac{Y}{N} - 1} > \epsilon \right) < \epsilon \right\}.$$
\end{proof}

For a set $S \subseteq \mathbb{R}^k$ and a real number $\alpha \in \mathbb{R}$, the set $\alpha S$ is the set of all scalar products (of elements of $S$) by $\alpha$. That is,
$$\alpha S \defeq \{y \in \mathbb{R}^k: y = \alpha x \text{ for some } \alpha \in A\}.$$

For $S \subseteq \mathbb{R}^k$ and $A \subseteq \mathbb{R}$, the set $AS$ is defined as the set of all scalar products of elements of $S$ with elements in $A$. That is,

\begin{align}\label{scalar product}
    AS \defeq \cup_{\alpha \in A} \alpha S.
\end{align}

Scalar products (with elements of ${^*}\mathbb{R}$ or with internal subsets of ${^*}\mathbb{R}$) are analogously defined in the nonstandard universe by transfer. We note the following elementary fact about small scalings of compact sets that will be useful in the sequel.

\begin{lemma}\label{lemma for compact sets}
Let $C$ be a compact subset of $\mathbb{R}^k$. Then we have:
\begin{align}\label{compact intersection result}
    \bigcap_{n \in \mathbb{N}_{>1}} \left[1 - \frac{1}{n}, 1+ \frac{1}{n} \right]C  = C.
\end{align}
\end{lemma}

\begin{proof}
Let the left side of \eqref{compact intersection result} be called $\tilde{C}$ for brevity. It is clear that $C \subseteq \tilde{C}$. To show the inclusion from the other side, consider $x \in \tilde{C}$. Thus for each $n \in \mathbb{N}_{>1}$, there exist $\alpha_n \in \mathbb{R}$ and $y_n \in C$ such that $x = \alpha_n y_n$. By the sequential compactness of $C$, find a subsequence $(n_k)_{k \in \mathbb{N}}$ such that $\lim_{k \rightarrow \infty} y_{n_k}$ exists as an element of $C$. Say, $\lim_{k \rightarrow \infty} y_{n_k} = y \in C$. Note that, by construction, we have $\lim_{k \rightarrow \infty} \alpha_{n_k} = 1$. By continuity of the scalar product map, we thus have the following:
\begin{align}
    x = \lim_{k \rightarrow \infty} \alpha_{n_k} y_{n_k} = \left(\lim_{k \rightarrow \infty} \alpha_{n_k}\right)\left(\lim_{k \rightarrow \infty} y_{n_k}\right) = y \in C,
\end{align}
completing the proof. 
\end{proof}

We now prove Poincar\'e's theorem that we restate here for convenience:
\begin{reptheorem}{Poincare's theorem}
For all bounded measurable functions $f\co \mathbb{R}^k \to \mathbb{R}$, we have:
\begin{align*}
\lim_{n \rightarrow \infty} \int_{S^{n-1}(\sqrt{n})} f d\bar{\sigma} = \int_{\mathbb{R}^k} f d\mu.    
\end{align*}\end{reptheorem}

\begin{proof}
Let $B$ be a Borel subset of $\mathbb{R}^k$ and let $\mbx = (X_1, \ldots, X_N)$ be as defined in \eqref{definition of X}. For $k \in \mathbb{N}$, let $\mbx_{(k)}$ be the projection $(X_1, \ldots, X_k)$ onto ${^*}\mathbb{R}^k$. Using \eqref{pushforward characterization} and taking standard parts on both sides yields the following:
\begin{align}
    L\bar{\sigma} (\{(x_1 \ldots, x_N) \in S^{N-1}(\sqrt{N}): (x_1, \ldots, x_{k}) \in {^*}B\}) &= L{^*}\mathbb{P} \left( \frac{\sqrt{N}\mbx_{(k)}}{\sqrt{Y}} \in {^*}B\right) \nonumber \\
    &= L{^*}\mathbb{P} \left( \mbx_{(k)} \in \sqrt{\frac{Y}{N}}{^*}B\right).
\end{align}

Using Lemma \ref{N over theta lemma}, the last expression is less than or equal to 
$$L{^*}\mbp \left( \mbx_{(k)} \in \bigcap_{n = 2}^m {\mystrut^{*}}\left[1 - \frac{1}{n}, 1+ \frac{1}{n} \right] B\right) = L{^*}\mbp \left(\mbx_{(k)} \in {\mystrutt^*}\left(\bigcap_{n = 2}^m \left[1 - \frac{1}{n}, 1+ \frac{1}{n} \right]  B\right)\right)$$
for all $m \in \mathbb{N}$. 

Taking limits as $m \rightarrow \infty$, we obtain:
\begin{align}
 &L\bar{\sigma} (\{(x_1 \ldots, x_N) \in S^{N-1}(\sqrt{N}): (x_1, \ldots, x_{k}) \in {^*}B\}) \nonumber \\
 &\leq \lim_{m \rightarrow \infty} L{^*}\mbp \left( \mbx_{(k)} \in {\mystrutt^*}\left[\bigcap_{n = 2}^m \left[1 - \frac{1}{n}, 1+ \frac{1}{n} \right]  B\right)\right) \nonumber \\
 &= \lim_{m \rightarrow \infty} \mbp \left( \mbx_{(k)} \in \bigcap_{n = 2}^m \left[1 - \frac{1}{n}, 1+ \frac{1}{n} \right] B\right) \nonumber \\
 &= \mbp \left( \mbx_{(k)} \in \bigcap_{n \in \mathbb{N}_{>1}} \left[1 - \frac{1}{n}, 1+ \frac{1}{n} \right]  B\right). \label{Borel set inequality}
\end{align}

By \eqref{Borel set inequality} and Lemma \ref{lemma for compact sets}, we have the following inequality:
\begin{align*}
    L\bar{\sigma} (\{(x_1 \ldots, x_N) \in S^{N-1}(\sqrt{N}): (x_1, \ldots, x_{k}) \in {^*}C\}) \leq \mbp \left( \mbx_{(k)} \in C \right) = \mu_{(k)}(B)\\
    \text{for all compact subsets } C \subseteq \mathbb{R}^k.
\end{align*}

Since $N > \mathbb{N}$ is arbitrary and $\mu_{(k)}$ is a Radon measure, Proposition \ref{TFAE} and the nonstandard characterization of limits complete the proof. 
\end{proof}

\begin{section}{On the limiting behavior of a sequence of probability spaces}\label{section 3}
Toward the proof of Poincar\'e's theorem in the previous section, we showed that for an arbitrary $N > \mathbb{N}$, the surface area measure over $S^{N-1}(\sqrt{N})$ (which may be thought of as the $N^{\text{th}}$ element of the sequence of of spheres $(S^{n-1}(\sqrt{n}))_{n \in \mathbb{N}}$) assigns the same measure (up to infinitesimals) to fibers of finite dimensional sets as the Gaussian measures of such sets (in their respective ambient Euclidean spaces). This idea is explored in more abstract settings in the current section in order to generalize to limiting results for integrals of unbounded functions. 

\subsection{Integrating finite dimensional functions along nice sequences of probability spaces}
Let $\{(\Omega_n, \mathcal{F}_n, \nu_n)\}_{n \in \mathbb{N}}$ be a sequence of probability spaces. Viewing the sequence as a function on $\mathbb{N}$, we get an internal probability space $(\Omega_N, \mathcal{F}_N, \nu_N)$ for each $N > \mathbb{N}$. Note that we have been dropping the ${^*}$ when it is clear from context that the index $N$ is hyperfinite. Philosophically, the Loeb space $(\Omega_N, L(\mathcal{F}_N), L\nu_N)$ for $N > \mathbb{N}$ should capture the long-term behavior of the sequence $\{(\Omega_n, \mathcal{F}_n, \nu_n)\}_{n \in \mathbb{N}}$ of probability spaces. We will often omit the sigma-algebra when there is no chance of confusion. Drawing inspiration from Theorem \ref{S-integrable TFAE}\ref{S4}, we obtain the following theorem in this regard.

\begin{theorem}\label{sufficient condition for integrability}
Let $(E, \mathcal{E})$ be a measure space. Let $k \in \mathbb{N}$ and for each $n \in \mathbb{N}_{> k}$, suppose $\Omega_n \subseteq E^{n'}$ for some $n' \in \mathbb{N}_{> k}$. Suppose that $\mathcal{F}_n$, the given sigma-algebra on $\Omega_n$, is induced by the product sigma-algebra $\mathcal{E}_{n'}$ on $E^{n'}$. Let $(\Omega_n, \mathcal{F}_n, \nu_n)$ be a sequence of Borel probability spaces. Let $f\co E^k \to \mathbb{R}$ satisfy 
\begin{align}\label{double}
   \lim_{m \rightarrow \infty} \lim_{n \rightarrow \infty} \int_{\Omega_n \cap \{\abs{f} \geq m\}} \abs{f} d{\nu_n} = 0. 
\end{align}
Then, $f$ is integrable over $\Omega_n$ for large $n $, so that the sequence $\alpha_{f, n} \defeq \int_{\Omega_n} f d{\nu_n}$ is well-defined for large $n$. Furthermore, for any $N > \mathbb{N}$, the function $\st({^*}f)$ is Loeb integrable over $(\Omega_N, L(\mathcal{F}_N), L\nu_N)$ and satisfies
\begin{align*} 
          \st(\alpha_{f, N}) = \int_{\Omega_{N}} \st({^*}f) d{L\nu_N}. 
\end{align*}
\end{theorem}

\begin{remark}
Bounded measurable functions trivially satisfy the hypothesis in (\ref{double}).
\end{remark}

\begin{proof}
For a fixed $\epsilon \in \mathbb{R}_{>0}$, there exists $\ell_{\epsilon} \in \mathbb{N}$ such that the following holds:  for any $m \geq \ell_\epsilon$, there is an $n_{\epsilon, m} \in \mathbb{N}$ such that for all $n \geq n_{\epsilon, m}$, we have \begin{align}\label{tightness type equation}
    \int_{\Omega_n \cap \{\abs{f} \geq m\}} \abs{f} d{\nu_n} < \epsilon.
\end{align}

In particular, $f$ is integrable on $\Omega_n$ for all $n > n_{\epsilon, \ell_\epsilon}$, with the integral of the absolute value being at most $(\ell_\epsilon + \epsilon)$. Further, for any $M, N > \mathbb{N}$, transfer yields
$$\starint_{\Omega_N} \abs{{^*}f}\mathbbm{1}_{\{\abs{{^*}f} > M\}} d \nu_N  \leq \starint_{\Omega_N} \abs{{^*}f}\mathbbm{1}_{\{\abs{{^*}f} > {\ell_\epsilon}\}} d \nu_N < \epsilon \text{ for all } \epsilon \in \mathbb{R}_{>0}.$$

Given $N > \mathbb{N}$, ${^*}f$ is $S$--integrable on $\Omega_N$ by Theorem \ref{S-integrable TFAE}\ref{S2}. 

Now, $\alpha_{f, N}$ is the ${^*}$--integral of ${^*}f$ over $(\Omega_N, \nu_N)$ by transfer. Note that $$f = f_+ - f_-,$$
where $f_+ \defeq \max\{f, 0\}$ and $f_- \defeq \max\{-f, 0\}$. By transfer, we then have:
\begin{align}\label{alpha for f+ and f-}
\alpha_{f, N} = \alpha_{f_+, N} - \alpha_{f_-, N}.
\end{align}

Since ${^*}f$ is $S$--integrable on $(\Omega_N, \nu_N)$, so are ${^*}f_+$ and ${^*}f_-$ (this is because $\abs{{^*}f_+} \text{ and } \abs{{^*}f_-}$ are at most equal to $\abs{{^*}f}$). Since ${^*}f_+$ and ${^*}f_-$ are nonnegative functions, Theorem \ref{S-integrable TFAE}\ref{S4} implies:
\begin{align}\nonumber
    \alpha_{f_+, N} &= \int_{\Omega_N} \st({^*}f_+) dL\nu_N, \\
    \text{and }  \alpha_{f_-, N} &= \int_{\Omega_N} \st({^*}f_-) dL\nu_N. \label{positive and negative parts}
\end{align}

Using this in \eqref{alpha for f+ and f-} and then using the fact that $\st({^*}f)$ is Loeb integrable completes the proof.
\end{proof}

\begin{corollary}\label{consequences}
Let $(E, \mathcal{E})$ be a measure space. Let $k \in \mathbb{N}$ and for each $n \in \mathbb{N}_{> k}$, suppose $\Omega_n \subseteq E^{n'}$ for some $n' \in \mathbb{N}_{> k}$. Suppose that $\mathcal{F}_n$, the given sigma-algebra on $\Omega_n$, is induced by the product sigma-algebra $\mathcal{E}_{n'}$ on $E^{n'}$. Let $(\Omega_n, \mathcal{F}_n, \nu_n)$ be a sequence of Borel probability spaces. Let $\mathbb{P}$ be a probability measure on $(E^k, \mathcal{E}_k)$ such that
$L\nu_N ({^*}B) = \mathbb{P}(B) \text{ for any } B \in \mathcal{E}_k \text{ and } N > \mathbb{N}$.
\begin{enumerate}[(i)]
    \item\label{measurable} If $f\co E^k \to \mathbb{R}$ is measurable, then 
    $$L\nu_N(\{x \in \Omega_N: {^*}f(x) \in {^*}\mathbb{R}_{\text{fin}}\}) = 1 \text{ for all } N > \mathbb{N}.$$
    \item\label{bounded measurable} If $f\co E^k \to \mathbb{R}$ is bounded and measurable, then $$\lim_{n \rightarrow \infty} \int_{\Omega_n} f d\nu_n = \int_{E^k} f d\mathbb{P} = \int_{\Omega_N} \st({^*}f) dL\nu_N \text{ for all } N > \mathbb{N}.$$ 
    \item\label{integrable} If $f\co E^k \to \mathbb{R}$ is $\mathbb{P}$--integrable, then we have that $\st({^*}f)$ is $L\nu_N$--integrable for all $N > \mathbb{N}$. Furthermore, for any $N > \mathbb{N}$, we have:
    $$\int_{E^k}f d\mathbb{P} = \int_{\Omega_N} \st({^*}f) dL\nu_N \text{, and } \int_{E^k}\abs{f} d\mathbb{P} = \int_{\Omega_N} \abs{\st({^*}f)} dL\nu_N$$
\end{enumerate}
\end{corollary}
\begin{proof}
\ref{measurable} follows from Corollary \ref{functions are finite almost everywhere}. \ref{bounded measurable} follows from Theorem \ref{sufficient condition for integrability}, Corollary \ref{from bounded to integrable} and the nonstandard characterization of limits. Finally, \ref{integrable} follows from Corollary \ref{from bounded to integrable}, completing the proof.
\end{proof}

Note that Corollary \ref{consequences}\ref{integrable} allows us to express the expected value of a $\mathbb{P}$--integrable function $f\co E^k \to \mathbb{R}$ as the Loeb integral of $\st({^*}f)$ over $\Omega_{N}$ for all hyperfinite $N$. However, this does not necessarily imply that the sequence $\alpha_{f,n} \defeq \int_{\Omega_n} f d\nu_n$ converges to $\int_{E^k} f d\mathbb{P}$, as $\alpha_{f, N}$ may not be infinitesimally close to the Loeb integral of $\st({^*}f)$ over $\Omega_N$ in general. To see a counterexample, consider $(E, \mathcal{E}) = (\mathbb{N}_0, \mathcal{P}(\mathbb{N)}))$ (where $\mathbb{N}_0 = \mathbb{N} \cup \{0\}$), with $\Omega_n \defeq \{0, n\}$ for each $n \in \mathbb{N}$. Define $\mathbb{P} \defeq \mathbb{1}_{\{0\}}$, the probability measure concentrated at $0$. Define $\nu_n(\{0\}) = 1 - \frac{1}{n}$ and $\nu_n(\{n\}) = \frac{1}{n}$. Then for any $N > \mathbb{N}$, the Loeb measure $L\nu_N$ assigns full mass to $\{0\}$. Thus the hypotheses of Corollary \ref{consequences} are satisfied. Consider the measurable function $f \co \mathbb{N}_0 \to \mathbb{R}$ defined by $f(n) \defeq n$ for all $n \in \mathbb{N}$. It is clear that $\alpha_{f, N}$ equals $1$ while the Loeb integral of $\st({^*}f)$ equals $0$. 

In view of Theorem \ref{S-integrable TFAE}, the correct criterion needed for $\alpha_{f, N}$ to be infinitesimally close to the Loeb integral of $\st({^*}f)$ over $\Omega_N$ for nonnegative functions $f$ is the $S$--integrability of ${^*}f$ over $\Omega_N$. This also means that the sufficient criterion \eqref{double} in Theorem \ref{sufficient condition for integrability} is necessary if we restrict to nonnegative functions. We record and prove these observations in the following theorem.

\begin{theorem}\label{main equivalence}
In the setting of Corollary \ref{consequences}, the following are equivalent for a nonnegative function $f\co E^k \to \mathbb{R}_{\geq 0}$:
\begin{enumerate}
    \item\label{good limit} $f$ is $\mathbb{P}$--integrable and $\lim_{n \rightarrow \infty} \int_{\Omega_n} f d\nu_n = \int_{E^k} f d\mathbb{P}$.
    \item\label{S-integrable} The nonstandard extension ${^*}f$ is $S$--integrable on $\Omega_N$ for all $N > \mathbb{N}$.
    \item\label{good double limit} The function $f$ is integrable on $(\Omega_n, \nu_n)$ for all large $n \in \mathbb{N}$, and furthermore:
    $$ \lim_{m \rightarrow \infty} \lim_{n \rightarrow \infty} \int_{\Omega_n \cap \{{f} \geq m\}} {f} d{\nu}_n = 0.$$
\end{enumerate}
\end{theorem}
\begin{proof} $\boldbox{\mathit{(1) \Rightarrow (2)}}$

Assume that $f$ is $\mathbb{P}$--integrable and $\lim_{n \rightarrow \infty} \int_{\Omega_n} f d\nu_n = \int_{E^k} f d\mathbb{P}$. Using the nonstandard characterization of limits, Corollary \ref{consequences}[\ref{integrable}], and Theorem \ref{S-integrable TFAE}\ref{S4} (making use of the fact that $f = \abs{f}$ since $f$ is assumed to be nonnegative), it follows that ${^*}f$ is $S$--integrable on $\Omega_N$ for any $N > \mathbb{N}$.

 $\boldbox{\mathit{(2) \Rightarrow (3)}}$ 
 
 Now assume that ${^*}f$ is $S$--integrable on $\Omega_N$ for all $N > \mathbb{N}$. As a consequence (using either Theorem \ref{S-integrable TFAE}\ref{S2} or Theorem \ref{S-integrable TFAE}\ref{S3}), we have that ${^*}f\mathbbm{1}_{\{\abs{{^*}f} \geq m\}}$ is $S$--integrable on $\Omega_N$ for any $N > \mathbb{N}$ and $m \in \mathbb{N}$. Fix $N_0 > \mathbb{N}$ such that the following is true (existence of such an $N_0$ is guaranteed by the nonstandard characterization of limit superiors):
\begin{align*}
   \limsup_{n \rightarrow \infty} \int_{\Omega_n \cap \{\abs{f} \geq m\}} \abs{f} d{\nu}_n &= \st \left(\starint_{\Omega_{N_0}} {^*}\abs{f}\mathbbm{1}_{\{{^*}\abs{f} \geq m\}} d{\nu}_{N_0} \right).
\end{align*}

By Theorem \ref{S-integrable TFAE}\ref{S4}, we get:
\begin{align}
    \limsup_{n \rightarrow \infty} \int_{\Omega_n \cap \{\abs{f} \geq m\}} \abs{f} d{\nu}_n &= \int_{\Omega_{N_0}} \st \left({^*}\abs{f} \mathbbm{1}_{\{ {^*}\abs{f} \geq m \}} \right) dL{\nu}_{N_0} \nonumber \\
    \Rightarrow \lim_{m \rightarrow \infty} \limsup_{n \rightarrow \infty} \int_{\Omega_n \cap \{\abs{f} \geq m\}} \abs{f} d{\nu}_n &= \lim_{m \rightarrow \infty} \int_{\Omega_{N_0}} \st \left({^*}\abs{f} \mathbbm{1}_{\{ {^*}\abs{f} \geq m \}} \right) dL{\nu}_{N_0}. \label{N_0}
\end{align}

Since ${^*}f$ is $S$--integrable on $\Omega_{N_0}$, it follows that $\st({^*}f)$ is Loeb integrable on $\Omega_{N_0}$. Hence the limit on the right side of (\ref{N_0}) is zero, as desired.

$\boldbox{\mathit{(3) \Rightarrow (1)}}$ 

This follows from Theorem \ref{sufficient condition for integrability}, Corollary \ref{consequences}\ref{integrable}, and Theorem \ref{S-integrable TFAE}\ref{S4}.
\end{proof}

\subsection{Application to a proof of the Riemann-Lebesgue Lemma}
The theory of limiting integrals built over the last two subsections may theoretically be applied to a lot of situations in which the probability spaces are changing. While we will cover its application to spherical integrals in the next section, we include here a new proof of the famous Riemann--Lebesgue lemma as an illustration of the versatility of this theory. We paraphrase the Riemann--Lebesgue lemma below (see, for example, Rudin \cite[5.14, p. 103]{Rudin-RnC}).
\begin{theorem}[Riemann--Lebesgue Lemma]\label{Riemann-Lebesgue}
Let $\lambda$ be the Lebesgue measure on the interval $T \defeq [-\pi, \pi]$. If $f \in L^1(T, \lambda)$, then we have:
$$\lim_{n \rightarrow \infty} \int_T f(x) \cos(nx) d\lambda(x) = 0 \text{ and } \lim_{n \rightarrow \infty} \int_T f(x) \sin(nx) d\lambda(x) = 0.$$
\end{theorem}

\begin{proof}
For each $n \in \mathbb{N}$, define $g_n \co T \to \mathbb{R}$ by $g_n(x) = \frac{1 - \cos(nx)}{2\pi}$. The functions $g_n$ are probability densities on $[-\pi, \pi]$. For each $n \in \mathbb{N}$, let $\mathbb{P}_n$ denote the probability measure on $T$ with the density $g_n$. By integrating the densities for $n \in \mathbb{N}$, we find that the corresponding probability distribution functions are given by:
$$G_n(x) \defeq \mathbb{P}_n\{(- \infty, x]\} = \frac{1}{2\pi}\left(x - \frac{\sin(x)}{n} \right) \text{ for all } x \in T.$$

As $n \rightarrow \infty$, the sequence $G_n$ converges pointwise to the distribution function of the uniform (normalized) Lebesgue measure $\mathbb{P}$ on $[-\pi, \pi]$. Thus $\mathbb{P}_n \overset{weak}\longrightarrow \mathbb{P}$, that is, 
$$\lim_{n \rightarrow \infty} \int_T f d\mathbb{P}_n = \int_{T} f d\mathbb{P} \text{ for all bounded continuous } f \co T \rightarrow \mathbb{R}.$$ 

By an equivalent criterion for weak convergence, we obtain:
\begin{align}\label{first weak inequality for open sets}
     \liminf_{n \rightarrow \infty} \mathbb{P}_n(U) \geq \mathbb{P}(U) \text{ for all open subsets } U \subseteq T.
\end{align}

By the nonstandard characterization of limit inferiors, this is equivalent to:
\begin{align}\label{first inequality for open sets '}
    \mathbb{P}_N({^*}U) \geq \mathbb{P}(U) \text{ for all open subsets } U \subseteq T \text{ and } N > \mathbb{N}.
\end{align}

Since the density function $g_n$ for $\mathbb{P}_n$ is pointwise bounded above by the density function for $\mathbb{P}$, by transfer we also obtain the other side of the above inequality. That is, we obtain:

\begin{align}\label{second inequality for open sets '}
    \mathbb{P}_N({^*}U) \leq \mathbb{P}(U) \text{ for all open subsets } U \subseteq T \text{ and } N > \mathbb{N}.
\end{align}

Combining \eqref{first inequality for open sets '} and \eqref{second inequality for open sets '}, we obtain:

\begin{align}\label{weak convergence equality}
    \mathbb{P}_N({^*}U) = \mathbb{P}(U) \text{ for all open subsets } U \subseteq T \text{ and } N > \mathbb{N}.
\end{align}

By Proposition \ref{TFAE}, we obtain:
\begin{align}\label{bounded functions over T}
    \int\limits_{{^*}T} \st({^*}f) dL\mathbb{P}_N = \int\limits_{T} f d\mathbb{P} \text{ for all bounded measurable } f\co T \rightarrow \mathbb{R} \text{ and } N > \mathbb{N}.
\end{align}  

For any $f \in L^1(T, \lambda)$, we use the facts that $\abs{g_n} \leq \frac{1}{\pi}$ and $f \in L^1(T, \lambda)$ to get:
\begin{align}\label{double limit over T}
    \lim_{m \rightarrow \infty}\lim_{n \rightarrow \infty} \int_T \abs{f}\mathbbm{1}_{\abs{f} > m} d\mathbb{P}_n(x) \leq \frac{1}{\pi} \lim_{m \rightarrow \infty} \int_T \abs{f}\mathbbm{1}_{\abs{f} > m} d\lambda(x) = 0.
\end{align}

Using \eqref{bounded functions over T} and \eqref{double limit over T} in Theorem \ref{main equivalence} (with $(T, \mathbb{P}_n)$ playing the role of $(\Omega_n, \nu_n)$ in that theorem), we obtain, for each $f \in L^1(T, \lambda) = L^1(T, \mathbb{P})$:
\begin{align*}
    &\lim_{n \rightarrow \infty} \int_T f(x) d\mathbb{P}_n(x) = \int_T f(x) d\mathbb{P}(x) \\
    \Rightarrow  &\lim_{n \rightarrow \infty} \int_T \left(\frac{f(x)}{2\pi}  - \frac{f(x) \cos(nx)}{2\pi}\right) d\lambda(x) = \int_T \frac{f(x)}{2\pi} d\lambda(x)\\
   \Rightarrow &\lim_{n \rightarrow \infty} \int_T f(x) \cos(nx) d\lambda(x) = 0.
\end{align*}

The proof for $\sin(nx)$ goes exactly the same way if we replace the $f_n$ by the probability density functions $g_n(x) = \frac{1 - \sin(nx)}{2\pi}$ for $x \in T$.
\end{proof}

\subsection{What happens if the finite dimensional function is not nice in the limiting space?} In general, for a function $f\co E^k \to \mathbb{R}$ (not necessarily satisfying the conditions in Theorem \ref{main equivalence}), the following result allows us to still approximate its integral by a suitably modified sequence of integrals over $(\Omega_n, \nu_n)$. Note that this result is in the spirit of Littlewood's three principles from measure theory (see \cite[p.26]{littlewood}) -- approximating a potentially ill-behaved integrable function by well-behaved bounded functions.

\begin{lemma}\label{approximation}
In the setting of Corollary \ref{consequences}, let $f\co E^k \to \mathbb{R}$ be $\mathbb{P}$--integrable. Given any $\epsilon, \delta, \theta \in \mathbb{R}_{>0}$ there exist an $n_0 \in \mathbb{N}$ and functions $g_n \co \Omega_n \to \mathbb{R}$ for all $n \in \mathbb{N}_{\geq n_0}$ such that the following hold:
\begin{enumerate}[(i)]
    \item\label{g_n integrable} $\abs{g_n}$ is bounded by $n$ for all $n \in \mathbb{N}_{\geq n_0}$.
    \item\label{g_n close to f} $\nu_n\left( \abs{g_n - f} > \delta \right) < \epsilon$ for all $n \in \mathbb{N}_{\geq n_0}$.
    \item\label{integral of g_n is close} $\abs{\int_{\Omega_n}g_n d\nu_n - \int_{E^k} f d\mathbb{P}} < \theta$ for all $n \in \mathbb{N}_{\geq n_0}$.
\end{enumerate}
\end{lemma}

\begin{proof}
By Corollary \ref{consequences}\ref{integrable}, we know that 
$$\int_{E^k}\abs{f} d\mu = \int_{\Omega_N} \st({^*}\abs{f}) dL\nu_N \text{ for all } N > \mathbb{N}.$$

Thus, for any $N> \mathbb{N}$, the map $\st({{^*}f})$ is Loeb integrable on $\Omega_N$, and hence has an $S$--integrable lifting $G_N \co \Omega_N \rightarrow {^*}\mathbb{R}$ by Theorem \ref{S-integrable lifting}. In particular, 
\begin{gather}
    L\nu_N\left(\st(G_N) = \st({^*}f)\right) = 1 \label{g_N 1} \text{, and } \\
    \st\left(\starint_{\Omega_N} G_N d\nu_N \right) = \int_{\Omega_N} \st({G_N}) dL\nu_N = \int_{\Omega_N} \st({^*}f) dL\nu_N = \int_{E^k} f d\mathbb{P}. \label{g_N 2}
\end{gather}

Equation \eqref{g_N 1} follows from the definition of lifting. The first equality in \eqref{g_N 2} follows from Theorem \ref{S-integrable TFAE}\ref{S4} applied to the nonnegative $S$--integrable functions $(G_N)_+ \defeq \max\{G_N, 0\}$ and $(G_N)_- \defeq \max \{-G_N, 0\}$. The second equality in \eqref{g_N 2} follows from equation \eqref{g_N 1}, while the last equality in \eqref{g_N 2} follows from Corollary \ref{consequences}[\ref{integrable}]. 

Without loss of generality, we can assume that $\abs{G_N} \leq N$ for all $N > \mathbb{N}$ (as we may replace $G_N$ by the function $G_N \mathbbm{1}_{\abs{G_N} \leq N}$, which still satisfies \eqref{g_N 1} and \eqref{g_N 2}). Thus, for the given $\epsilon, \delta, \theta \in \mathbb{R}_{>0}$, the following internal set contains ${^*}\mathbb{N} \backslash \mathbb{N}$.
\begin{align*}
    \mathcal{G}_{\epsilon, \delta, \theta} \defeq \Bigg\{ n &\in {^*}\mathbb{N}: ~\exists G_n \in {^*}L^1(\Omega_n, \nu_n) \text{ such that } ~\abs{G_n} \leq n, \\
    &{^*}\nu_n\left( \abs{G_n - {^*}f} > \delta \right) < \epsilon, \text{ and } \abs{\starint_{{^*}\Omega_n} G_n d{^*}\nu_n - \int_{E^k} f d\mathbb{P}} < \theta \Bigg\}.
\end{align*}

By underflow, we find $n_0 \in \mathbb{N}$ such that $\mathbb{N}_{\geq n_0} \subseteq \mathcal{G}_{\epsilon, \delta, \theta}$. Now fix an $n \in \mathbb{N}_{\geq n_0}$. In the nonstandard universe, the following statement is true:
\begin{align*}
    &\exists G_n \in {^*}L^1(\Omega_n, \nu_n) \\
    &\left((\abs{G_n} \leq n) \land \left( {^*}\nu_n\left( \abs{G_n - {^*}f} > \delta \right) < \epsilon) \right) 
    \land \left(\abs{\starint_{{^*}\Omega_n} G_n d{^*}\nu_n - \int\limits_{E^k} f d\mathbb{P}} < \theta \right)\right).
\end{align*}
Transfer of this sentence yields a $g_n \in L^1(\Omega_n, \nu_n)$ with the desired properties. \qedhere


\end{proof}

We can strengthen Lemma \ref{approximation} as follows, by requiring the functions to have the same domain $E^k$.

\begin{theorem}\label{approximation corollary}
In the setting of Corollary \ref{consequences}, let $f\co E^k \to \mathbb{R}$ be $\mathbb{P}$--integrable. Given any $\epsilon, \delta, \theta \in \mathbb{R}_{>0}$ there exist an $n_0 \in \mathbb{N}$ and functions $g_n \co E^k \to \mathbb{R}$ for all $n \in \mathbb{N}_{\geq n_0}$ such that the following hold:
\begin{enumerate}[(i)]
    \item\label{E^k 1} $\abs{g_n}$ is bounded by $n$ for all $n \in \mathbb{N}_{\geq n_0}$.
    \item\label{E^k 2} $\nu_n\left( \abs{g_n - f} > \delta \right) < \epsilon$ for all $n \in \mathbb{N}_{\geq n_0}$.
    \item\label{E^k 3} $\abs{\int_{\Omega_n}g_n d\nu_n - \int_{E^k} f d\mathbb{P}} < \theta$ for all $n \in \mathbb{N}_{\geq n_0}$.
\end{enumerate}
\end{theorem}
\begin{proof}
For $n \in \mathbb{N}_{\geq k}$, define $\nu_n' \co \mathcal{E}_k \to [0,1]$ by $\nu_n'(B) = \nu_n((B \times E^{n-k}) \cap \Omega_n)$. For any bounded measurable $g\co E^k \to \mathbb{R}$, expressing $g$ as a uniform limit of simple functions yields 
\begin{align}\label{measure extension}
    \int_{\Omega_n} g d\nu_n = \int_{E^k} g d\nu_n'. 
\end{align}
Let $(g_n)_{n \in \mathbb{N}}$ be a sequence of functions obtained by applying Lemma \ref{approximation} to the sequence $(E^k, \nu_n')_{n \in \mathbb{N}}$ of probability spaces. Then \ref{E^k 1}, \ref{E^k 2} and \ref{E^k 3} follow from the corresponding results in Lemma \ref{approximation} together with (\ref{measure extension}).
\end{proof}
\end{section}

\begin{section}{Generalizing Poincar\'e's theorem}\label{section 4}
\subsection{Revisiting a standard proof of Poincar\'e's theorem}\label{computational proof}
For the rest of the paper, we let $S_n$ denote the sphere $S^{n-1}(\sqrt{n})$ and $\bar{\sigma}_n$ denote $\bar{\sigma}_{S_n}$, for all $n \in \mathbb{N}$. Fix $k \in \mathbb{N}$ and let $\mu$ denote the standard k-dimensional Gaussian measure. Let $B_k(a)$ denote the open ball of radius $a$ in $\mathbb{R}^k$. For a set $B \in \mathcal{B}(\mathbb{R}^k)$ and any $n \in \mathbb{N}_{\geq k}$, we define $\bar{\sigma}_n(B)$ to be the value of $\bar{\sigma}_n(\{x \in S_n : \pi_{k}(x) \in B\}) = \bar{\sigma}_n\left((B \times \mathbb{R}^{n-k}) \cap S_n\right)$, where $\pi_{k}$ is the projection onto $\mathbb{R}^k$. Similarly, a function $f\co \mathbb{R}^k \to \mathbb{R}$ is canonically extended to $\mathbb{R}^n$ by using `$f(x, y)$' to denote $f(x)$ for all $x \in \mathbb{R}^k$ and $y \in \mathbb{R}^{n-k}$. 

In an attempt to generalize Theorem \ref{Poincare's theorem}, we first look at another proof of the same result using classical analysis. This proof requires directly evaluating the spherical integrals and using dominated convergence theorem (compare with the less computational proof of Theorem \ref{Poincare's theorem} in Section \ref{nonstandard Poincare}). We again rephrase Theorem \ref{Poincare's theorem} below. 

\begin{reptheorem}{Poincare's theorem}
For all bounded measurable functions $f\co \mathbb{R}^k \to \mathbb{R}$, we have:
\begin{align*}
\lim_{n \rightarrow \infty} \int_{S^{n-1}(\sqrt{n})} f d\bar{\sigma} = \int_{\mathbb{R}^k} f d\mu.    
\end{align*}\end{reptheorem}

\begin{proof}
Let $\lambda$ denote the Lebesgue measure on $\mathbb{R}^k$. By Sengupta's disintegration formula (see Sengupta \cite[Proposition 4.1]{Sengupta}), we have the following chain of equalities for any bounded measurable $f \co \mathbb{R}^k \to \mathbb{R}$.

\begin{flalign}
     \nonumber &\int\limits_{S^{n-1}(\sqrt{n})} f d\bar{\sigma}_n \\
     \nonumber &= \frac{1}{\sigma(S_n)} \int\limits_{x \in B_k(\sqrt{n})} \int_{y \in S^{n - k - 1} (\sqrt{n - \norm{x}^2})} f(x, y) d\sigma(y) \frac{\sqrt{n}}{\sqrt{n - \norm{x}^2}} d\lambda(x) \\
    \nonumber &= \frac{1}{\sigma(S_n)} \int_{\mathbb{R}^k} 
    \sigma\left(S^{n - k - 1} \left(\sqrt{n - \norm{x}^2}\right)\right) \cdot \frac{\mathbbm{1}_{B_k(\sqrt{n})}(x) f(x) \sqrt{n}}{\sqrt{n - \norm{x}^2}} d\lambda(x) \\
     \nonumber &= \frac{\Gamma\left(\frac{n}{2}\right)}{2 \pi^{\frac{n}{2}} \cdot (\sqrt{n})^{n-1}} \int\limits_{\mathbb{R}^k} \frac{2 \pi^{\frac{n - k}{2}} (n - \norm{x}^2)^{\frac{n - k - 1}{2}}}{\Gamma\left(\frac{n - k}{2}\right)} \cdot \frac{\mathbbm{1}_{B_k(\sqrt{n})}(x) f(x) \sqrt{n}}{\sqrt{n - \norm{x}^2}} d\lambda(x) \\
  &= \label{disintegration} a_{n, k} b_{n, k} \int_{\mathbb{R}^k} \frac{1}{{(\sqrt{2 \pi})}^k} \left(1 - \frac{\norm{x}^2}{n} \right)^{\frac{n}{2}} \frac{\mathbbm{1}_{B_k(\sqrt{n})}(x)f(x)}{\left(1 - \frac{\norm{x}^2}{n} \right)^{\frac{k+2}{2}}} d\lambda(x), 
    \end{flalign}
where $a_{n, k} = \frac{\Gamma\left(\frac{n}{2}\right)}{\Gamma\left(\frac{n - k}{2}\right) \cdot \left(\frac{n-k}{2}\right)^{\frac{k}{2}}} \text{ and } b_{n, k} = \left(1 - \frac{k}{n}\right)^{\frac{k}{2}}.$

Note that $\lim_{n \rightarrow \infty} a_{n, k} = \lim_{n \rightarrow \infty} b_{n, k} = 1$ for all $k \in \mathbb{N}$ (the first limit following from Stirling's formula, see Rudin \cite[equation 103, p. 194]{Rudin}). 

Modulo constants, for large values of $n$, the integrand in \eqref{disintegration} is bounded by $\abs{f(x)} e^{-\frac{\norm{x}^2}{4}}$, which is integrable on $\mathbb{R}^k$ since $f$ is assumed to be bounded. Thus by the dominated convergence theorem, the integral in (\ref{disintegration}) converges to $\int_{\mathbb{R}^k} f d\mu$ as $n \rightarrow \infty$, as desired.
\end{proof}

\begin{remark}\label{remark on dct not working}
Due to the factor of $\left(1 - \frac{\norm{x}^2}{n} \right)^{\frac{k+2}{2}}$ in the denominator of \eqref{disintegration}, dominated convergence theorem does not directly work when we work with an unbounded function $f$, as there is no reason for $\abs{f(x)} e^{-\frac{\norm{x}^2}{4}}$ to be Lebesgue integrable in general. Indeed for a general Gaussian integrable $f$, we can bound $\abs{f(x)}\left(1 - \frac{\norm{x}^2}{n} \right)^{\frac{n}{2}}$ by $\abs{f(x)} e^{-\frac{\norm{x}^2}{2}}$, but there is still no obvious way to bound the whole integrand in \eqref{disintegration} by a Lebesgue integrable function due to that extra factor in the denominator.
\end{remark}

\begin{corollary}\label{finite coordinates}
For $k \in \mathbb{N}$ and $N > \mathbb{N}$, almost all points on $S_N$ have finite first $k$ coordinates. That is,
\begin{align*}
        L\bar{\sigma}_N(\{(x_1, \ldots, x_N) \in S^{N-1}(\sqrt{N}): x_1, \ldots, x_k \in {^*}\mathbb{R}_{\text{fin}}\}) = 1.
    \end{align*}
\end{corollary}
\begin{proof}
Fix $k$ and $N$ as above. If $m \in \mathbb{N}$, we have $L\bar{\sigma}_N ({^*}(-m,m)^k) = \mu((-m, m)^k)$ by Theorem \ref{Poincare's theorem}. Letting $m \rightarrow \infty$ on both sides completes the proof.
\end{proof}

\begin{corollary}\label{equator 2}
For any $t \in \mathbb{R}_{>1}$, we have  $$\lim_{n \rightarrow \infty}\int_{\{x \in \mathbb{R}^k: \frac{n}{t} < \norm{x}^2 < n\}} \left(1 - \frac{\norm{x}^2}{n} \right)^{\frac{n}{4}} d\lambda(x) = 0.$$
\end{corollary}
\begin{proof}
Let $t \in \mathbb{R}_{>1}$ and $N > \mathbb{N}$. As a consequence of Corollary \ref{finite coordinates}, we obtain:
\begin{align*}
    \bar{\sigma}_N \left(\left\{ x \in S^{N-1}(\sqrt{N}): \frac{N}{t} < \norm{\pi_{k}(x)}^2 < N \right\}\right) \approx 0.
\end{align*}

The nonstandard characterization of limits and equation  (\ref{disintegration}) thus yield the following.
\begin{align}\label{specific theta}
    \lim_{n \rightarrow \infty}\int_{\{x \in \mathbb{R}^k: \frac{n}{t} < \norm{x}^2 < n\}} \frac{1}{{(\sqrt{2 \pi})}^k} \left(1 - \frac{\norm{x}^2}{n} \right)^{\frac{n}{2}} \frac{1}{\left(1 - \frac{\norm{x}^2}{n} \right)^{\frac{k+2}{2}}} d\lambda(x) = 0.
\end{align}
For all $n \in \mathbb{N}_{\geq 2(k+2)}$, the sequence in the statement of the Corollary is bounded above by (a constant times) the sequence in (\ref{specific theta}), thus completing the proof.
\end{proof}

\begin{remark}
We can also prove Corollary \ref{equator 2} directly by noting that $$\left(1 - \frac{\norm{x}^2}{n}\right)^{\frac{n}{4}}\mathbbm{1}_{\norm{x}^2 \leq n} \leq e^{-\frac{\norm{x}^2}{4}},$$ where the right side is Lebesgue integrable over $\mathbb{R}^k$. The proof presented above is still valuable because it exposes a connection between these integrals and surface area measures.  
\end{remark}

\subsection{A useful inequality between spherical and Gaussian measures}
In this subsection, we derive an inequality comparing the $L^1$ norm (over the sphere $S^{n-1}(\sqrt{n})$) of a function defined on $\mathbb{R}^k$ and its $p^\text{th}$ moment (for any $p \in \mathbb{R}_{>1}$) with respect to the standard Gaussian measure on $\mathbb{R}^k$. 

With the foresight provided by the philosophy of spherical integrals being close to a Gaussian integral, we expect these spherical integrals to be asymptotically bounded by the $L^p(\mathbb{R}^k, \mu)$--norms as the dimensions increase. Theorem \ref{p>1} shows that depending on the value of $p \in \mathbb{R}_{>1}$, there is a dimension (namely $4(k+1)q$) beyond which this does happen. Before we prove that theorem, we need to generalize Sengupta's disintegration formula to work for any nonnegative function.

\begin{theorem}\label{generalization of Theorem of Sengupta}
Let $N$ and $k$ be positive integers with $k < N$. Suppose $f$ is either a bounded measurable or a nonnegative measurable function on $S^{N-1}(a)$, the sphere in $\mathbb{R}^N = \mathbb{R}^k \times \mathbb{R}^{N-k}$ of radius $a$ and with center $0$. Then, with $\sigma$ denoting surface measure (non-normalized) on spheres, \begin{align}\label{Sengupta's formula generalized}
    \int_{z \in S^{N-1}(a)} {f(z)} d\sigma(z) = \int_{x \in B_k(a)}\left( \int_{y \in S^{N-k-1}(a_x)} {f(x, y)} d\sigma(y)\right) \frac{a}{a_x} dx
\end{align}
for any $a \in \mathbb{R}_{>0}$, where $a_x = \sqrt{a^2 - \norm{x}^2}$. The above equality means that either both sides are finite and equal, or both sides are infinite. 
\end{theorem}
\begin{proof}
If $f$ is bounded measurable, then this is just Sengupta's disintegration formula (see Sengupta \cite[Proposition 4.1]{Sengupta}). Otherwise, if $f$ is nonnegative, then apply Sengupta's disintegration formula to the bounded measurable functions $f_m \defeq {f}\cdot \mathbbm{1}_{{f} \leq m}$ for each $m \in \mathbb{N}$, and then use monotone convergence theorem on both sides to obtain \eqref{Sengupta's formula generalized}.
\end{proof}

\begin{theorem}\label{p>1}
For each $p \in \mathbb{R}_{>1}$, there is a constant $C_p \in \mathbb{R}_{>0}$ such that 
\begin{align}\label{L2 inequality}
    \int_{S^{n-1}(\sqrt{n})} \abs{g} d\bar{\sigma}_n \leq C_p[\mathbb{E}_\mu(\abs{g}^p)]^{\frac{1}{p}} \text{ for all } g \in L^p(\mathbb{R}^k, \mu) \text{ and } n \in \mathbb{N}_{> 4(k+2)q},
\end{align}
where $q \in \mathbb{R}_{>0}$ is such that $\frac{1}{p} + \frac{1}{q} = 1$.
\end{theorem}

\begin{proof}
Fix $g \in L^p(\mathbb{R}^k, \mu)$, where $p \in \mathbb{R}_{>1}$. Also, let $t \in \mathbb{N}_{>1}$. Using Theorem \ref{generalization of Theorem of Sengupta} instead of Sengupta \cite[Proposition 4.1]{Sengupta}, we can follow the same steps leading up to (\ref{disintegration}) to see that  $\int_{S^{n-1}(\sqrt{n})} \abs{g} d\bar{\sigma}_n$ is equal to
\begin{align}
    &\int_{\norm{x}^2 \leq \frac{n}{t}} \frac{a_{n, k} b_{n, k}\abs{g(x)}}{{(\sqrt{2 \pi})}^k} \left(1 - \frac{\norm{x}^2}{n} \right)^{\frac{n-k-2}{2}} d\lambda(x)   \nonumber \\
     + &\int_{\frac{n}{t} < \norm{x}^2 \leq n} \frac{a_{n, k} b_{n, k}\abs{g(x)}}{{(\sqrt{2 \pi})}^k} \left(1 - \frac{\norm{x}^2}{n} \right)^{\frac{n}{2}} \frac{1}{\left(1 - \frac{\norm{x}^2}{n} \right)^{\frac{k+2}{2}}} d\lambda(x),  \label{pre-Holder}
\end{align}
where $a_{n,k} =  \frac{\Gamma\left(\frac{n}{2}\right)}{\Gamma\left(\frac{n - k}{2}\right) \cdot \left(\frac{n-k}{2}\right)^{\frac{k}{2}}}$ and $b_{n, k} = {\left( 1 - \frac{k}{n} \right)}^{\frac{k}{2}}$ are the same constants that appear in \eqref{disintegration}. 

Note that $${\left(1 - \frac{\norm{x}^2}{n} \right)^{-\frac{k+2}{2}}} \leq \left(\frac{t}{t-1}\right)^{\frac{k+2}{2}} \text{ whenever } \norm{x}^2 \leq \frac{n}{t}.$$

Also, $\left(1 - \frac{\norm{x}^2}{n}\right)^{\frac{n}{2}} \mathbbm{1}_{\norm{x}^2 \leq n}$ is at most equal to $e^{-\frac{\norm{x}^2}{2}}$ for all $x \in \mathbb{R}^k$. Noting that $b_{n, k} < 1$ for all $n \in \mathbb{N}_{>k}$, the first summand in (\ref{pre-Holder}) is at most
$$\left(\frac{t}{t-1}\right)^{\frac{k+2}{2}} \frac{a_{n,k}}{(2\pi)^{\frac{k}{2}}} \int_{\mathbb{R}^k} \abs{g({x})} e^{-\frac{\norm{{x}}^2}{2}} d\lambda({x})$$
for all $n \in \mathbb{N}_{>k}$. Writing this integral as a Gaussian expected value, and then using Jensen's inequality, we have:

\begin{align}\label{first term for Lp}
   I_1 \leq a_{n,k}  \left(\frac{t}{t-1}\right)^{\frac{k+2}{2}} \norm{g}_{L^p(\mathbb{R}^k, \mu)} \text{ for all } n \in \mathbb{N}_{>k},
\end{align}
where $I_1$ is the $\text{first summand in \eqref{pre-Holder}}$, and $\norm{g}_{L^p(\mathbb{R}^k, \mu)} = (\mathbb{E}_{\mu}(\abs{g}^p))^{\frac{1}{p}}$.

Let $q \in \mathbb{R}_{>1}$ be such that $\frac{1}{p} + \frac{1}{q} = 1$. Then we can write the second summand in \eqref{pre-Holder} as follows:
\begin{align}\label{second summand 0}
    a_{n, k} b_{n, k}\int_{\frac{n}{t} < \norm{x}^2 \leq n} \frac{\abs{g(x)}}{{(\sqrt{2 \pi})}^\frac{k}{p}} \left(1 - \frac{\norm{x}^2}{n} \right)^{\frac{n}{2p}} \cdot \frac{1}{(\sqrt{2\pi})^{\frac{k}{q}}} {\left(1 - \frac{\norm{x}^2}{n} \right)^{\frac{n}{2q} - \frac{k+2}{2}}} d\lambda(x).
\end{align}

Note that $b_{n, k} < 1$ for all $n \in \mathbb{N}_{>k}$. By H\"older's inequality applied to the functions $x \mapsto \frac{\abs{g(x)}}{{(\sqrt{2 \pi})}^\frac{k}{p}} \left(1 - \frac{\norm{x}^2}{n} \right)^{\frac{n}{2p}}$ and $x \mapsto \frac{1}{(\sqrt{2\pi})^{\frac{k}{q}}} {\left(1 - \frac{\norm{x}^2}{n} \right)^{\frac{n}{2q} - \frac{k+2}{2}}}$ (on the domain $\{x \in \mathbb{R}^k: \frac{n}{t} < \norm{x}^2 < n\}$ equipped with its Lebesgue measure), the expression in \eqref{second summand 0} is at most equal to the following \footnote{An anonymous referee has pointed out that one could also apply H\"older's inequality to the functions $x \mapsto \abs{g(x)}$ and $x \mapsto \left(1 - \frac{\norm{{x}}^2}{n} \right)^{- \frac{k + 2}{2}}$, on the same domain but with the measure given by $d\nu(x) = \frac{1}{(\sqrt{2\pi})^k} \left(1 - \frac{\norm{x}^2}{n} \right)^{\frac{n}{2}} d\lambda(x)$.}: 
\begin{align*}
   &a_{n,k} \left(\int_{\substack{{x} \in \mathbb{R}^{k} \\ \frac{n}{t} < \norm{{x}}^2 \leq n}} \abs{{g({x})}}^p \cdot \frac{1}{{(\sqrt{2 \pi})}^{k}} \left(1 - \frac{\norm{{x}}^2}{n} \right)^{\frac{n}{2}} d\lambda({x}) \right)^{\frac{1}{p}} \\
  \times &\left(\int_{\substack{{x} \in \mathbb{R}^{k} \\ \frac{n}{t} < \norm{{x}}^2 \leq n}} \frac{1}{{(\sqrt{2 \pi})}^{k}} \left(1 - \frac{\norm{{x}}^2}{n} \right)^{\left(\frac{n}{2q} - \frac{k + 2}{2}\right)\cdot q} d\lambda({x}) \right)^{\frac{1}{q}}.\\
\end{align*}

The first term in this product is at most $a_{n,k}({\mathbb{E}_{\mu}(\abs{g}^p)})^{\frac{1}{p}}$. Also, the integrand in the second term in this product is at most $\left(1 - \frac{\norm{{x}}^2}{n} \right)^{\frac{n}{4}}$ for all $n \in \mathbb{N}_{>2(k + 2)q}$. To summarize, if $I_2$ is the \text{second summand in \eqref{pre-Holder}}, then we have:

\begin{align}\label{second term for Lp}
   I_2 \leq a_{n,k} \norm{g}_{L^p(\mathbb{R}^k, \mu)} \cdot  \theta_{n, t} \text{ for all } n \in \mathbb{N}_{>(k+2)q},
\end{align}
where 
\begin{align}\label{theta}
    \theta_{n,t} = \left(\int_{\substack{{x} \in \mathbb{R}^{k} \\ \frac{n}{t} < \norm{{x}}^2 \leq n}} \frac{1}{{(\sqrt{2 \pi})}^{k}} \left(1 - \frac{\norm{{x}}^2}{n} \right)^{\frac{n}{4} } d\lambda({x}) \right)^{\frac{1}{q}}.
\end{align}

Combining \eqref{first term for Lp} and \eqref{second term for Lp}, we get:

\begin{align}\label{pre Lp inequality}
  \int_{S^{n-1}(\sqrt{n})} \abs{g} d\bar{\sigma}_n \leq & a_{n, k} \left[ \left(\frac{t}{t-1}\right)^{\frac{k+2}{2}} + \theta_{n, t}\right] \norm{g}_{L^p(\mathbb{R}^k, \mu)} \end{align}
\text{ for all } $n \in \mathbb{N}_{>4(k+2)q}$ and $t \in \mathbb{N}_{>1}$.

Here $a_{n,k} =  \frac{\Gamma\left(\frac{n}{2}\right)}{\Gamma\left(\frac{n - k}{2}\right) \cdot \left(\frac{n-k}{2}\right)^{\frac{k}{2}}}$ and $\theta_{n,t}$ is as in \eqref{theta}. Note that $\lim_{n \rightarrow \infty} a_{n, k} = 1$, and by Corollary \ref{equator 2}, $\lim_{n \rightarrow \infty} \theta_{n, t} = 0$ for all $t \in \mathbb{N}$. Thus, for any $t \in \mathbb{N}$, the coefficient of $\norm{g}_{L^p(\mathbb{R}^k, \mu)}$ in \eqref{pre Lp inequality} is uniformly bounded above, by (say) $C_p$. This completes the proof of the theorem.
\end{proof}

Focusing on the coefficient in \eqref{pre Lp inequality}, we note that given $\epsilon \in \mathbb{R}_{>0}$ we can choose $t \in \mathbb{N}_{>1}$ large enough for which the following inequality holds. $$\left( \frac{t}{t - 1} \right)^{\frac{k+2}{2}} < 1 + \frac{\epsilon}{2}.$$

For this $t$, using Corollary \ref{equator 2}, we can choose an $n_p \in \mathbb{N}$ large enough such that $\theta_{n, t} < \frac{\epsilon}{2}$ for all $n \in \mathbb{N}_{>n_p}$. Since $\lim_{n \rightarrow \infty} a_{n, k} = 1$, we can also ensure that the $n_p$ we choose is large enough such that $a_{n, k} < 1 + \epsilon$ for all $n \in \mathbb{N}_{>n_p}$. Combining all of this, \eqref{pre Lp inequality} yields the following useful corollary: we are able to bound the ratio of the spherical integral and the Gaussian $L^p$ norm by a constant as close to $1$ as we want, with the price of having to go to a potentially higher dimension to observe this phenomenon. 

\begin{corollary}\label{p>1 part 2}
For each $p \in \mathbb{R}_{>1}$ and $\epsilon \in \mathbb{R}_{>0}$, there is an $n_p \in \mathbb{N}$ such that 
\begin{align}\label{L2 inequality part 2}
    \int_{S^{n-1}(\sqrt{n})} \abs{g} d\bar{\sigma}_n \leq (1 + \epsilon)[\mathbb{E}_\mu(\abs{g}^p)]^{\frac{1}{p}} \text{ for all } g \in L^p(\mathbb{R}^k, \mu) \text{ and } n \in \mathbb{N}_{> n_p}.
\end{align}
\end{corollary}

Using Theorem \ref{p>1}, the condition \eqref{double} of Theorem \ref{sufficient condition for integrability} is easily verified for all functions in $L^p(\mathbb{R}^k, \mu)$, where $p \in \mathbb{R}^k$. Using that theorem and Theorem \ref{Poincare's theorem}, we obtain our main limiting result for spherical integrals.

\begin{theorem}\label{most general}
If $\mu$ is the standard Gaussian measure on $\mathbb{R}^k$ and $f \in L^p(\mathbb{R}^k, \mu)$ for some $p \in \mathbb{R}_{>1}$,  then the nonstandard extension ${^*}f$ is $S$--integrable on $S^{N-1}(\sqrt{N})$ for all $N > \mathbb{N}$. As a consequence, the function $f$ is integrable on $(S^{n-1}(\sqrt{n}), \bar{\sigma}_n)$ for all large $n \in \mathbb{N}$, and
    $$\lim_{m \rightarrow \infty} \lim_{n \rightarrow \infty} \int_{S^{n-1}(\sqrt{n}) \cap \{\abs{f} \geq m\}} \abs{f} d{\bar{\sigma}}_n = 0.$$
    
Furthermore, the spherical integrals of $f$ satisfy the following limiting behavior:
$$\lim_{n \rightarrow \infty} \int_{S^{n-1}(\sqrt{n})} f d\bar{\sigma}_n = \int_{\mathbb{R}^k} f d\mu.$$

This limit of spherical integrals can be written as a single spherical integral (over an infinite sphere) $\int_{S^{N-1}(\sqrt{N})} \st({^*}f) dL\bar{\sigma}_N$ for any hyperfinite $N$.
\end{theorem}
\end{section}

\appendix\section{The kinetic theory of gases and spherical surface measures}\label{appendix}
This appendix is devoted to the physical motivation behind viewing a high-dimensional spherical integral as a Gaussian mean. We will give an outline of the usual derivation of the Maxwell--Boltzmann distribution (originally discovered by Maxwell in \cite{Maxwell} and improved by Boltzmann in \cite{Boltzmann}), and explain its connection with the problem on limiting spherical integrals studied in this paper. We recommend Chapter 5 of Pauli and Enz \cite{pauli} (which we also roughly follow for our outline) for more details on the underlying physics. 

We work under the assumption that a statistically large number (which we shall denote by $N$) of particles of a monatomic gas are moving randomly in a container of a given volume. Each particle has a mass $m$. We further assume that the velocity of a given particle behaves like a random vector following an isotropic continuous probability density function $f \co \mathbb{R}^3 \to \mathbb{R}$, where the isotropicity of $f$ just means the following:

\begin{align}\label{isotropic}
    \exists ~ g\co \mathbb{R} \to \mathbb{R} \text{ such that } f(v_1, v_2, v_3) = g({v_1}^2 + {v_2}^2 + {v_3}^2) \text{ for all } v_1, v_2, v_3 \in \mathbb{R}.
\end{align}

Newtonian mechanics can be used to postulate that the pressure on any wall of the container is directly proportional to the mean squared speed of the gas particles. Combining this with the ideal gas law, it then follows that the average kinetic energy of the particles should be directly proportional to the temperature $T$ of the system. This is typically described by the following equation, where $\vec{v}_i$ is the velocity of the $i^{\text{th}}$ particle, and $k$ is a constant called the \textit{Boltzmann constant}. Note that the factor of $\frac{3}{2}$ appears in the following in order to make sure that our $k$ agrees with the traditional value of the Boltzmann constant. 

\begin{align}\label{average energy}
    \sum_{i = 1}^{N} \frac{1}{2}m \norm{\vec{v}_i}^2 = \frac{3}{2}kTN, \text{ that is}, ~ \frac{\sum_{i = 1}^{N} \norm{\vec{v}_i}^2}{N} = \frac{3kT}{m}.
\end{align}

We also assume that the three components of the velocity vector of a given particle are independent and identically distributed, with a continuous density function $h \co \mathbb{R} \to \mathbb{R}$ satisfying the following condition.
\begin{align}\label{independent}
     f(v_1, v_2, v_3) = h(v_1)h(v_2)h(v_3) \text{ for all } v_1, v_2, v_3 \in \mathbb{R}.
\end{align}

We define new functions $\psi \co \mathbb{R} \to \mathbb{R}$ and $\phi\co \mathbb{R}_{\geq 0} \to \mathbb{R}$ by the following formulae:

\begin{align*}
\psi(v_i) &\defeq \log(h(v_i)) \text{ for all } v_i \in \mathbb{R}, \text{ and }\\
\phi(v^2) &\defeq \log (g(v^2)) \text{ for all } v \in \mathbb{R}.
\end{align*}

Then $\phi$ and $\psi$ satisfy the following functional equation:
\begin{align}\label{functional equation}
    \phi({v_1}^2 + {v_2}^2 + {v_3}^2) = \psi(v_1) + \psi(v_2) + \psi(v_3).
\end{align}

Assuming that $\phi$ and $\psi$ are sufficiently differentiable, it can be shown that \eqref{functional equation} can be satisfied only if $\phi$ is linear. After some simplifications, we obtain the following.

\begin{align}\label{alpha and C}
    f(v_1, v_2, v_3) = g({v_1}^2 + {v_2}^2 + {v_3}^2) = Ce^{-\alpha ({v_1}^2 + {v_2}^2 + {v_3}^2)},
\end{align}
for some constants $C, \alpha \in \mathbb{R}_{>0}$. 

The constant $C$ is obtained to be $\left( \frac{\alpha}{\pi} \right)^{\frac{3}{2}}$ by integrating both sides of \eqref{alpha and C} and noting that the integral of $f$ is equal to $1$ as $f$ is a probability density function. We then compute the expected value of the square of the speed ${v_1}^2 + {v_2}^2 + {v_3}^2$, and equate it with $\frac{3kT}{m}$ (which comes \eqref{average energy}, using our underlying hypothesis of $N$ being statistically large so that the mean of the individual particles' squared speed should be very close to the theoretical expected value -- more precisely, one can let $N \rightarrow \infty$ and use the strong law of large numbers). From that, we find $\frac{3}{2\alpha} = \frac{3kT}{m}$, so that $\alpha = \frac{m}{2kT}$. We thus obtain the famous Maxwell--Boltzmann distribution for velocity:

\begin{align}\label{Maxwell-Boltzmann equation}
    f(v_1, v_2, v_3) = \left(\frac{m}{2\pi k T} \right)^{\frac{3}{2}}e^{-\frac{m}{2kT}({v_1}^2 + {v_2}^2 + {v_3}^2)} \text{ for all } v_1, v_2, v_3 \in \mathbb{R}.
\end{align}

From the above formula, Maxwell and Boltzmann proceeded to derive probability distributions of other important functions (such as speed) of velocity. These distributions are heavily used in statistical mechanics and thermodynamics.

The problem of statistically estimating the behavior of a function of the velocity of a random gas particle can be reinterpreted in a useful way with the notion of surface area measures on Euclidean spheres. For simplicity of terms we let $N_0 \defeq 3N$, and renormalize the constants in equation \eqref{average energy} (by assuming that ${kT} = {m}$). Writing $\vec{v}_i = (v_{i,x}, v_{i,y}, v_{i,z}) \in \mathbb{R}^3$, we then get:
\begin{align}\label{Gas equation 2}
    \sum_{i = 1}^{N} \left({v_{i,x}}^2 + {v_{i,y}}^2 + {v_{i,z}}^2\right) = N_0.
\end{align}

Hence $(\vec{v}_1, \ldots, \vec{v}_N)$ is a vector in $\mathbb{R}^{N_0}$ of norm $\sqrt{N_0}$. In other words, $(\vec{v}_1, \ldots, \vec{v}_N)$ is an element of $S^{N_0-1}(\sqrt{N_0})$. Since we do not have any information about the motion of these particles other than what is contained in equation \eqref{average energy}, it is reasonable to assume that the value of $(\vec{v}_1, \ldots, \vec{v}_N)$ at a given time is a ``random point'' of $S^{N_0-1}(\sqrt{N_0})$. The surface area measure $\bar{\sigma}_S$ for a sphere $S$ serves as a notion of a uniform probability measure on $S$. Thus we can make the observation regarding $(\vec{v}_1, \ldots, \vec{v}_N)$ being a random point of $S^{N_0-1}(\sqrt{N_0})$ more precise by postulating that the probability that $(\vec{v}_1, \ldots, \vec{v}_N)$ lies in a Borel set $B \subseteq S^{N_0-1}(\sqrt{N_0})$ is given by $\bar{\sigma}_{S^{N_0-1}(\sqrt{N_0})}(B)$. 

Since we are working under the assumption that the number of particles is very large, the probability that the first component of the velocity of the first particle, and hence of a random particle (by symmetry), is in a Borel set $B_1 \subseteq \mathbb{R}^1$, should be given by $\lim_{N_0 \rightarrow \infty}\bar{\sigma}((B_1 \times \mathbb{R}^{N_0-1}) \cap S^{N_0-1}(\sqrt{N_0}))$. Also, the expected or mean value of the first component of its velocity should be given by the following integral: $$\lim_{N_0 \rightarrow \infty} \int_{S^{N_0-1}(\sqrt{N_0})} v_{1,x} d\bar{\sigma}(v_{1,x}, v_{1,y}, v_{1,z}, \ldots, v_{N,x}, v_{N,y}, v_{N,z}).$$ 

Similarly, the expected value of speed would be given by the limit of the integrals of $\sqrt{v_{1,x}^2 + v_{1,y}^2 + v_{1,z}^2}$. In fact, the limit of integrals of any finite-dimensional function on these spheres can be interpreted as the expected value of some function of velocities of randomly chosen particles in our gaseous system. 

If there were a way to directly compute these limits, then we would be able to evaluate various probabilities associated with values taken by the velocity components, as well as recover the expected values of many functions of velocities of the particles. Furthermore, such a derivation would have the benefit of being less circular as we would not be making any assumptions on the nature (or even existence) of the density $f$ that was derived in \eqref{Maxwell-Boltzmann equation}. 

Thus the problem of generalizing Theorem \ref{Poincare's theorem} to the largest class of functions possible is intimately connected to, and has implications on our understanding of the kinetic theory of gases. Furthermore, the fact that there already exist distributions for functions of velocity such as speed (which, being equal to $\sqrt{{v_1}^2 + {v_2}^2 + {v_3}^2}$, is clearly not a bounded function) suggests that \eqref{Poincare's theorem} should, in principle, be generalizable to at least some unbounded functions, which in turn makes the problem of finding all such functions naturally appealing. 

Mathematically, \eqref{Poincare limit} tells us that the Gaussian measure $\mu$ is well-equipped to measure the limiting expected value of any bounded measurable function of a given collection of coordinates. In some sense, it retains all probabilistic information of the manner in which such functions behave over these spheres in the large-$N$ limit. From this point of view as well, it becomes a natural question to find out for which actual functions does it retain all such information. 

Nonstandard analysis gives access to hyperfinite natural numbers which provide a natural model for statistically large number of particles. The probability that the velocity of a random particle lies in some set could actually be thought of as the uniform surface area of the portion of a hyperfinite-dimensional sphere corresponding to this set.

\section*{Acknowledgments}
The author thanks Professor Ambar Sengupta for introducing the problem and participating in numerous mathematical discussions. The author is grateful to Professor Karl Mahlburg for feedback on the paper, and Professor Renling Jin for many fruitful discussions on nonstandard analysis. The author thanks Professor Arnab Ganguly for bringing to attention (while working on a different project) some properties of the density functions used in Theorem \ref{Riemann-Lebesgue}. The comments from an anonymous referee were helpful in fixing some errors and making the exposition more well-rounded.

\bibliography{References}
\bibliographystyle{amsplain}

\end{document}